\documentclass[a4paper,12pt]{amsart}
\usepackage{amsfonts}
\usepackage{amsmath,array}
\usepackage{amssymb}
\usepackage{mathrsfs}
\usepackage[colorlinks]{hyperref}
 \usepackage{graphicx}
\usepackage{enumerate}
\usepackage[usenames]{color}
\newtheorem{thm}{Theorem}[section]
\newtheorem{cor}[thm]{Corollary}
\newtheorem{lem}[thm]{Lemma}
\newtheorem{prop}[thm]{Proposition}

\numberwithin{equation}{section}
\usepackage[english]{babel} 
\usepackage{blindtext}

\theoremstyle{definition}
\newtheorem{definition}[thm]{Definition}
\newtheorem{rem}[thm]{Remark}

\DeclareMathOperator\ess{ess}
\begin{document}
 \title[Fourier multipliers on quantum Euclidean spaces]{$L^p$ -$L^q$ boundedness of Fourier multipliers on quantum Euclidean spaces}

\author[Michael Ruzhansky]{Michael Ruzhansky}
\address{
 Michael Ruzhansky:
  \endgraf
 Department of Mathematics: Analysis, Logic and Discrete Mathematics,
  \endgraf
 Ghent University, Ghent,
 \endgraf
  Belgium 
  \endgraf
  and 
 \endgraf
 School of Mathematical Sciences, Queen Mary University of London, London,
 \endgraf
 UK  
 \endgraf
  {\it E-mail address} {\rm michael.ruzhansky@ugent.be}
  }

\author[Serikbol Shaimardan]{Serikbol Shaimardan}
\address{
  Serikbol Shaimardan:
  \endgraf
  Institute of Mathematics and Mathematical Modeling, 050010, Almaty, 
  \endgraf
  Kazakhstan 
  \endgraf
  and 
  \endgraf
Department of Mathematics: Analysis, Logic and Discrete Mathematics
  \endgraf
 Ghent University, Ghent,
 \endgraf
  Belgium
  \endgraf
  {\it E-mail address} {\rm shaimardan.serik@gmail.com} 
  }

 \author[Kanat Tulenov]{Kanat Tulenov}
\address{
  Kanat Tulenov:
  \endgraf
  Institute of Mathematics and Mathematical Modeling, 050010, Almaty, 
  \endgraf
  Kazakhstan 
  \endgraf
  and 
  \endgraf
Department of Mathematics: Analysis, Logic and Discrete Mathematics
  \endgraf
 Ghent University, Ghent,
 \endgraf
  Belgium
  \endgraf
  {\it E-mail address} {\rm kanat.tulenov@ugent.be} 
  }

\date{}

\begin{abstract}
In this paper, we study Fourier multipliers on quantum Euclidean spaces and obtain results on their $L^p -L^q$ boundedness. On the way to get these results, we prove Paley, Hausdorff-Young-Paley, and Hardy-Littlewood inequalities on the quantum Euclidean space. As applications, we establish the $L^p -L^q$ estimate for the heat semigroup and Sobolev embedding theorem on quantum Euclidean spaces. We also obtain quantum analogues of logarithmic Sobolev and Nash type inequalities.
\end{abstract}

\subjclass[2020]{46L51, 46L52,  47L25, 11M55, 46E35, 42B15, 42B05, 43A50, 42A16, 32W30}

\keywords{Noncommutative Euclidean space, Fourier multipliers, Hausdorff-Young inequality, Sobolev embedding, heat semigroup.}

\maketitle

\tableofcontents
{\section{Introduction}}

 Quantum Euclidean spaces $\mathbb{R}^d_{\theta}$ defined in terms of an arbitrary antisymmetric real $d\times d$ matrix $\theta$, which are also known as Moyal spaces or quantum Euclidean spaces, are main objectives in the quantum geometry (see \cite{CGRS, green-book, GGSVM}), since they play the role of model noncompact spaces. In \cite{Rieffel}, Rieffel defined the quantum Euclidean space in terms of a deformation quantization of the classical Euclidean space $\mathbb{R}^d.$ These spaces were studied also by Moyal \cite{M} and Groenwald \cite{Gro} from diverse perspectives. Their aim is to deform the algebra of smooth functions on the classical Euclidean space $\mathbb{R}^d$ by replacing the pointwise product of functions with the twisted Moyal product. The so-called quantum-classical Weyl operators on $\mathbb{R}^d_{\theta}$ have been extensively studied in \cite{BW, DW}. Furthermore, the notion of convolution on $\mathbb{R}^d_{\theta}$, in the case where $\theta$ is a real invertible skew-symmetric matrix, was originally introduced in \cite{W}.
There are many equivalent constructions of quantum Euclidean spaces that appeared in the literature including defining them as the von Neumann algebra $L^{\infty}(\mathbb{R}^d_{\theta})$ with a twisted left-regular representation of $\mathbb{R}^d$ on the Hilbert space $L^2(\mathbb{R}^d).$ In this paper, our investigations are based on this definition. In the last few years, there has been substantial body of work on generalising the methods of the classical harmonic analysis on $\mathbb{R}^d$ to the quantum Euclidean space $\mathbb{R}^d_{\theta}.$ In particular, on the quantum Euclidean space $\mathbb{R}^d_{\theta},$ it is possible to define analogues of many of the tools of harmonic analysis, such as Fourier transform, differential operators and function spaces \cite{GJM,  GGSVM, GJP, HLW, LMcSZ, MSX, Mc, RShT-NoDEA, RShT-JFA}. For example, nonlinear partial differential equations and the Besov spaces were investigated in this quantum Euclidean space in a very recent paper \cite{Mc}. In this setting, many of classical inequalities in harmonic analysis and Fourier multipliers are seen in a new light. More precisely, in this paper we establish $L^p\to L^q$ boundedness of the Fourier multipliers in the quantum Euclidean space. To describe our main result, let us first recall the classical H\"ormander Fourier multiplier theorem settled in \cite[Theorem 1.11]{Hor}: for $1 < p \leq  2\leq q <\infty,$ the Fourier multiplier $g(D):\mathcal{S}(\mathbb{R}^d)\to\mathcal{S}(\mathbb{R}^d)$ with the symbol $g:\mathbb{R}^d \to \mathbb{C}$ defined by $\widehat{g(D)(f)}=g \cdot \widehat{f}, \quad f\in\mathcal{S}(\mathbb{R}^d),$ has a bounded extension from $L^p(\mathbb{R}^{d})$ to $L^q(\mathbb{R}^{d})$ if for the symbol $g$ the condition 
$$\sup\limits_{\lambda>0}\lambda\left(\int\limits_{x\in\mathbb{R}^d \atop  |g(x)|\geq\lambda}dx\right)^{\frac{1}{p}-\frac{1}{q}}<\infty$$
holds, where
$\widehat{f}$ is the Fourier transform
$$\widehat{f}(t)=(2\pi)^{-d/2}\int_{\mathbb{R}^d}f(s)e^{-i(t,s)}ds, \quad t\in \mathbb{R}^d,$$ and $(t,s)$ denotes the standard inner product of vectors $t$ and $s$ in $\mathbb{R}^d.$ The study of Fourier multipliers has been attracting attention of many researchers due to the fact that they have numerous applications in harmonic analysis, in particular, in PDEs. For the $L^p\to L^p$ boundedness of Fourier multipliers we refer the  reader to \cite{Hor, KM, Mikh, Mar, Stein}. 
There are also other works, which were devoted to the study of the $L^p \to L^q$ boundedness of the Fourier multipliers, see e.g. \cite{AMR, ARN, CK, Cow1993, Cow1974, K2, K3, K4, M-Nur, NT1, NT2, RV, RT-AdvOper, Z}. For example, in \cite{AR}, the authors obtained the $L^p\to L^q$ H\"ormander type theorem in the context of locally compact groups with applications such as the $L^p\to L^q$ boundedness of the heat semigroup and Sobolev embedding theorem in which our results are motivated.   
In this work, we study the $L^p\to L^q$ boundedness of the Fourier multipliers on the quantum Euclidean space. In other words, we obtain an analogue of the above mentioned H\"ormander Fourier multiplier theorem  on the quantum Euclidean space (see Theorem \ref{H-multiplier theorem-th}). In general, Fourier multipliers and their $L^p\to L^p$ and $L^p\to L^q$ boundedness were studied very recently in \cite{Mc} in terms of the so called Weyl transform. However, the condition for the symbol $g$ there is that the inverse Fourier transform of the symbol must  belong to $L^{1}(\mathbb{R}^d).$ Our condition is weaker than that of \cite[Theorem 3.2]{Mc}. On the other hand, our definition of the Fourier multiplier is more natural, and many properties are transferred similar to the commutative case. As in the classical case, the proof of the main result hinges upon the Paley inequality and Hausdorff–Young–Paley inequality for the Fourier transform obtained by using the Hausdorff–Young inequality. 
Therefore, on the way to achieve our goal, we obtain Paley, Hausdorff-Young-Paley, Hardy-Littlewood, and also the Logarithmic Sobolev inequalities on the quantum Euclidean space. We also show a simple prove of the $L^p\to L^q$ boundedness of the Fourier multipliers without using the Paley and Hausdorff-Young-Paley inequalities in our setting, where we used the idea in \cite{Z}.
In particular, we demonstrate several applications of the H\"ormander type theorem 
to obtain an $L^p\to L^q$ estimate for the heat semigroup (see Corollary \ref{heat-kernel}) and Sobolev type inequality (see Theorem \ref{sobolev-embed}) for $1<p\leq 2\leq q<\infty$ on the quantum Euclidean space. Finally, we proved logarithmic Sobolev and Nash type inequalities on the quantum Euclidean space.
Very recently, we have found a paper by H.J. Samuelsen \cite{HJS} which is closely related to our investigation, where the author studied Fourier-Wigner multipliers (which is an analog of the Fourier multiplier) in quantum phase space. The quantum phase space corresponds, up to unitary equivalence, to a special case of the quantum Euclidean space when $\theta$ is an invertible skew-symmetric matrix. Moreover, in this setting, the $L^p$-spaces over $\mathbb{R}^d_{\theta}$ are isomorphic to Schatten von Neumann classes, and there exist well-defined embeddings between $L^p$-spaces for varying values of $p$. Therefore, some results (for example, \cite[Proposition 2.6-2.7]{HJS}, \cite[Lemma 3.9.]{HJS}, and \cite[Theorem 3.13]{HJS}) reported in \cite{HJS} can be viewed as special case of our analysis in this work. A central distinction in our work is that we address the general situation where $\theta$ is not necessarily invertible. In this case, the usual inclusions fail in the general setting, which leads to substantial analytical differences. This distinction is crucial and creates additional difficulties when establishing inequalities and proving various results.

\;\;\;\;

\section{Preliminaries}
 
 \subsection{Quantum Euclidean space $L^{\infty}(\mathbb{R}^{d}_{\theta})$} \label{NC Euclidean space}
We refer the reader to \cite{GJM}, \cite{GGSVM}, \cite{LMcSZ}, \cite{MSX}, and \cite{Mc},  for a detailed discussion of the quantum  Euclidean space $\mathbb{R}^{d}_{\theta}$ and for additional information. 

Let $H$ be a Hilbert space. We denote by $B(H)$ the algebra of all bounded linear operators on $H.$ As usual $L^p(\mathbb{R}^d)$ ($1\leq p<\infty$) is the $L^p$-spaces of pointwise almost-everywhere equivalence classes of $p$-integrable functions and $L^{\infty}(\mathbb{R}^d)$ is the space of essentially bounded functions on the Euclidean space $\mathbb{R}^d.$  
For $1\leq p,q\leq\infty,$ the Lorentz space on $\mathbb{R}^d$ is the space of complex-valued measurable functions $f$ on $\mathbb{R}^d$ such that the following quasinorm is finite 
\begin{equation}\label{weak-Lp-commut}
\|f\|_{L^{p,q}(\mathbb{R}^d)}=\left\{ \begin{array}{rcl}
         \left(\int\limits_0^{\infty}\left(t^{\frac{1}{p}}\mu(t,f)\right)^q\frac{dt}{t}\right)^{\frac{1}{q}} & \mbox{for}
         & q<\infty; \\ \sup\limits_{t>0}t^\frac{1}{p}\mu(t,f)\;\;\;\;\;\;\;\;\;\;\;\; & \mbox{for} & q=\infty.  
                \end{array}\right. 
\end{equation}
where $\mu(\cdot,f)$ is the decreasing rearrangement of the function $f.$ For more details about these spaces, we refer the reader to \cite{G2008}. The next results will be used in Section \ref{5s}.
\begin{prop}\label{G-Prop} \begin{itemize}
\item[(i)] \cite[Proposition 1.4.10, p. 49]{G2008} Let $ 0 < p \leq \infty $ and $0 < q < r \leq\infty.$ Then there exists a constant $C_{p,q,r}>0$ (which depends on p, q, and r) such that
\begin{eqnarray}\label{Lorentz-embedding-1}
\|f\|_{L^{p,r}(\mathbb{R}^d)}\leq C_{p,q,r}\|f\|_{L^{p,q}(\mathbb{R}^d)}.    
\end{eqnarray} 
\item[(ii)] \cite[Exercises 1.4.19, p. 73]{G2008} Let $0<p,q,r\leq \infty,$ $0<s_1,s_2\leq\infty,$    $\frac{1}{p}+\frac{1}{q}=\frac{1}{r}$ and $\frac{1}{s_1}+\frac{1}{s_2}=\frac{1}{s}.$ Then there exists a constant $C_{p,q,s_1,s_2}>0$ (which depends on p, q, $s_1,$ and  $s_2$) such that,
\begin{eqnarray}\label{Lorentz-embedding-2}
\|fg\|_{L^{r,s}(\mathbb{R}^d)}\leq C_{p,q,s_1,s_2}\|f\|_{L^{p,s_1}(\mathbb{R}^d)}\|g\|_{L^{q,s_2}(\mathbb{R}^d)}.  
\end{eqnarray} 

\end{itemize}
\end{prop}

Given an integer $d\geq 1,$ fix an anti-symmetric $\mathbb{R}$-valued $d\times d$ matrix $\theta=\{\theta_{j,k}\}_{1\leq j,k\leq d}.$
\begin{definition}\label{NC_E_space1}Define $L^{\infty}(\mathbb{R}^{d}_{\theta})$ as the von Neumann algebra generated by the $d$-parameter strongly
continuous unitary family $\{U_{\theta}(t)\}_{t\in \mathbb{R}^{d}}$ satisfying the relation
\begin{equation}\label{weyl-relation}
U_{\theta}(t)U_{\theta}(s)=e^{\frac{1}{2}i(t,\theta s)}U_{\theta}(t+s),\quad t,s\in \mathbb{R}^d,
\end{equation}
where $(\cdot,\cdot)$ means the usual inner product in $\mathbb{R}^d.$
\end{definition}
The above relation is called the Weyl form of the canonical commutation relation. 
While it is possible to define $L^{\infty}(\mathbb{R}_{\theta}^d)$ in an abstract operator-theoretic way as in \cite{GGSVM}, we can also define the algebra as being generated by a concrete family of operators defined on the Hilbert space $L^2(\mathbb{R}^d).$  
    \begin{definition} \cite[Definition 2.1]{Mc}\label{NC_E_space2} For $t\in\mathbb{R}^d,$ denote by $U_{\theta}(t)$ the operator on $L^{2}(\mathbb{R}^{d})$ defined by the formula
    $$(U_{\theta}(t)\xi)(s)=e^{i(t,s)}\xi(s-\frac{1}{2}\theta t),\quad \xi\in L^{2}(\mathbb{R}^{d}),  \, t,s\in \mathbb{R}^d.$$
    It can be proved that the family $\{U_{\theta}(t)\}_{t\in \mathbb{R}^{d}}$ is strongly continuous and satisfies the relation \eqref{weyl-relation}.
    Then the von Neumann algebra $L^{\infty}(\mathbb{R}_{\theta}^d)$ is defined to be the weak operator topology closed subalgebra of $B(L^{2}(\mathbb{R}^{d}))$ generated by the family $\{U_{\theta}(t)\}_{t\in \mathbb{R}^{d}}$ and is called a n  quantum Euclidean space.
\end{definition}
Note that if $\theta=0,$ then above definitions reduce to the description of $L^{\infty}(\mathbb{R}^{d})$ as the
algebra of bounded pointwise multipliers on $L^{2}(\mathbb{R}^{d}).$
The structure of $L^{\infty}(\mathbb{R}_{\theta}^d)$ is determined by the Stone-von Neumann theorem, which asserts that if $\det(\theta)\neq 0,$ then any two $C^*$-algebras generated by a strongly continuous unitary family $\{U_{\theta}(t)\}_{t\in \mathbb{R}^{d}}$ satisfying Weyl relation \eqref{weyl-relation} are $*$-isomorphic \cite[Theorem 14.8]{H}. Therefore, the algebra of essentially bounded functions on $\mathbb{R}^{d}$ is recovered as a special case of the previous definition.

Generally, the algebraic nature of the quantum Euclidean space $L^{\infty}(\mathbb{R}_{\theta}^d)$ depends on the dimension of the kernel of $\theta.$ In the case $d=2,$ up to an orthogonal conjugation $\theta$ may be given as 
\begin{equation}\label{d=2}
    \theta=h\begin{pmatrix}
0 & -1\\
1 & 0
\end{pmatrix} 
\end{equation}
for some constant $h>0.$ In this case, $L^{\infty}(\mathbb{R}_{\theta}^2)$ is $*$-isomorphic to $B(L^{2}(\mathbb{R}))$ and this $*$-isomorphism can be written as 
$$U_{\theta}(t)\to e^{it_1\mathbf{M}_s +it_2 h \frac{d}{ds}},$$
where $\mathbf{M}_s\xi(s)=s\xi(s)$ and $\frac{d}{ds}\xi(s)=\xi'(s)$ is the differentiation.
If $d\geq 2,$ then an arbitrary $d\times d$ antisymmetric real matrix can be expressed (up to orthogonal conjugation) as a direct sum of a zero matrix and matrices of the form \eqref{d=2}, ultimately leading to the $*$-isomorphism
\begin{equation}\label{direct-sum}
    L^{\infty}(\mathbb{R}_{\theta}^d)\cong L^{\infty}(\mathbb{R}^{\dim(\ker(\theta))})\bar{\otimes}B(L^{2}(\mathbb{R}^{\text{rank}(\theta)/2})),
\end{equation}
where $\bar{\otimes}$ is the von Neumann tensor product \cite{MSX}. In particular, if $\det(\theta)\neq 0,$ then \eqref{direct-sum} reduces to
\begin{equation}\label{reduced-direct-sum}
    L^{\infty}(\mathbb{R}_{\theta}^d)\cong B(L^{2}(\mathbb{R}^{d/2})).
\end{equation}
Note that these formulas are meaningful since the rank of an anti-symmetric matrix is always even.
\subsection{Noncommutative integration} Let $f\in L^{1}(\mathbb{R}^d).$ Define $\lambda_{\theta}(f)$ as the operator defined by the formula
\begin{equation}\label{def-integration}
\lambda_{\theta}(f)\xi=\int_{\mathbb{R}^d}f(t)U_{\theta}(t)\xi dt, \quad \xi\in L^{2}(\mathbb{R}^d).
\end{equation}
This integral is absolutely convergent in the Bochner sense, and defines a bounded linear operator $\lambda_{\theta}(f): L^{2}(\mathbb{R}^d)\to L^{2}(\mathbb{R}^d)$ such that $\lambda_{\theta}(f)\in L^{\infty}(\mathbb{R}_{\theta}^d)$ (see \cite[Lemma 2.3]{MSX}). 
Let us denote by $\mathcal{S}(\mathbb{R}^d)$ the classical Schwartz space on $\mathbb{R}^d.$ For any $f\in\mathcal{S}(\mathbb{R}^d)$ we define the Fourier transform
$$\widehat{f}(t)=\int_{\mathbb{R}^d}f(s)e^{-i(t,s)}ds, \quad t\in \mathbb{R}^d.$$ The noncommutative Schwartz space $\mathcal{S}(\mathbb{R}_{\theta}^d)$ is defined as the image of the classical Schwartz space under $\lambda_{\theta},$ which is
$$\mathcal{S}(\mathbb{R}_{\theta}^d):=\{x\in L^{\infty}(\mathbb{R}_{\theta}^d):x=\lambda_{\theta}(f) \,\ \text{for some}\,\ f\in \mathcal{S}(\mathbb{R}^d)\}.$$
We define a topology on $\mathcal{S}(\mathbb{R}_{\theta}^d)$
 as the image of the canonical Fr\'{e}chet topology on $\mathcal{S}(\mathbb{R}^d)$ under $\lambda_{\theta}.$ 
 The topological dual of $\mathcal{S}(\mathbb{R}_{\theta}^d)$ will be denoted by $\mathcal{S}'(\mathbb{R}_{\theta}^d).$
The map $\lambda_{\theta}$ is injective \cite[Subsection 2.2.3]{MSX}, \cite{Mc}, and it can be extended to distributions. 
\subsection{Trace on $L^{\infty}(\mathbb{R}_{\theta}^d)$}
Given $f\in \mathcal{S}(\mathbb{R}^d),$ we
define the functional $\tau_{\theta}:\mathcal{S}(\mathbb{R}_{\theta}^d)\to \mathbb{C}$ by the formula
\begin{equation}\label{trace-def}
\tau_{\theta}(\lambda_{\theta}(f))=\tau_{\theta}\left(\int_{\mathbb{R}^d}f(\eta)U_{\theta}(\eta)d\eta\right):= f(0).
\end{equation}
Then, the functional $\tau_{\theta}$ admits an extension to a semifinite normal trace on $L^{\infty}(\mathbb{R}_{\theta}^d).$ Moreover, if $\theta=0,$ then $\tau_{\theta}$ is exactly the Lebesgue integral under a suitable isomorphism. If $\det(\theta)\neq0,$
then $\tau_{\theta}$ is (up to a normalisation) the operator trace on $B(L^2(\mathbb{R}^{d/2})).$ For more details, we refer to \cite{GJP}, \cite[Lemma 2.7]{MSX}, \cite[Theorem 2.6]{Mc}.

\subsection{Noncommutative $L^{p}(\mathbb{R}^{d}_{\theta})$ and $L^{p,q}(\mathbb{R}^{d}_{\theta})$ spaces}
With the definitions in the previous sections,  $L^{\infty}(\mathbb{R}^{d}_{\theta})$ is a semifinite von Neumann algebra with the trace $\tau_{\theta},$ and the pair $(L^{\infty}(\mathbb{R}^{d}_{\theta}),\tau_{\theta})$ is called  a noncommutative measure space. For any $1\leq p<\infty,$ we can define the $L^p$-norm on this space by the Borel functional calculus and the following formula:
$$\|x\|_{L^p(\mathbb{R}^d_{\theta})}=\Big(\tau_{\theta}(|x|^p)\Big)^{1/p},\quad x\in L^{\infty}(\mathbb{R}^{d}_{\theta}), $$
where $|x|:=(x^{*}x)^{1/2}.$
The completion of $\{x\in L^{\infty}(\mathbb{R}^{d}_{\theta}) :\|x\|_{p}<\infty\}$ with respect to $\|\cdot\|_{L^p(\mathbb{R}^d_{\theta})}$ is denoted by $L^p(\mathbb{R}^d_{\theta}).$ The elements of $L^p(\mathbb{R}^d_{\theta})$ are $\tau_{\theta}$-measurable operators like in the commutative case. These are linear densely defined closed (possibly unbounded) affiliated with $L^{\infty}(\mathbb{R}^{d}_{\theta})$ operators such that $\tau_{\theta}(\mathbf{1}_{(s,\infty)}(|x|))<\infty$ for some $s>0.$ Here, $\mathbf{1}_{(s,\infty)}(|x|)$ is the spectral projection with respect to the interval $(s,\infty).$ 
Let us denote the set of all 
$\tau_{\theta}$-measurable operators by $L^{0}(\mathbb{R}^{d}_{\theta}).$ 
Let $x=x^{\ast}\in L^{0}(\mathbb{R}^{d}_{\theta})$.  The {\it distribution function} of $x$ is defined by
$$n_x(s)=\tau_{\theta}\left(\mathbf{1}_{(s,\infty)}(x)\right), \quad -\infty<s<\infty.$$
For $x\in L^{0}(\mathbb{R}^{d}_{\theta}),$ the {\it generalised singular value function} $\mu(t, x)$ of $x$ is defined by
\begin{equation}\label{distribution-function}
\mu(t,x)=\inf\left\{s>0: n_{|x|}(s)\leq t\right\}, \quad t>0.
\end{equation}
The function $t\mapsto\mu(t,x)$ is decreasing and right-continuous. For more discussion on generalised singular value functions, we refer the reader to
\cite{FK, LSZ}.
The norm of $L^p(\mathbb{R}^d_{\theta})$ can also be written in terms of the generalised singular value function (see \cite[Example 2.4.2, p. 53]{LSZ}) as follows 
\begin{equation}\label{mu-norm}
\|x\|_{L^p(\mathbb{R}^d_{\theta})}=\left(\int_{0}^{\infty}\mu^{p}(s,x)ds\right)^{1/p}, \,\ \text{for} \,\ p<\infty\,\ \text{and} \,\,
\|x\|_{L^{\infty}(\mathbb{R}^d_{\theta})}=\mu(0,x), \,\, \text{for}\,\ p=\infty.
\end{equation}
The latter equality for $p=\infty,$ was proved in \cite[Lemma 2.3.12. (b), p. 50]{LSZ}.

The space $L^{0}(\mathbb{R}^{d}_{\theta})$ is a $*$-algebra, which can be made into a topological $*$-algebra as follows. Let
$$V(\varepsilon,\delta)=\{x\in L^{0}(\mathbb{R}^{d}_{\theta}): \mu(\varepsilon,x)\leq \delta\}.$$
Then $\{V(\varepsilon,\delta): \varepsilon,\delta>0\}$ is a system of neighbourhoods at 0 for which $L^{0}(\mathbb{R}^{d}_{\theta})$ becomes a metrizable topological $*$-algebra. The convergence with respect to this topology is called the {\it convergence in measure} \cite{PXu}. Next, we define the noncommutative Lorentz space associated with the NC Euclidean space.
\begin{definition} Let $1\leq p, q\leq \infty.$ Then, define the noncommutative Lorentz space $L^{p,q}(\mathbb{R}^{d}_{\theta})$ by
$$L^{p,q}(\mathbb{R}^{d}_{\theta}):=\{x\in L^{0}(\mathbb{R}^{d}_{\theta}):\|x\|_{L^{p,q}(\mathbb{R}^{d}_{\theta})}<\infty\},$$
where
\begin{equation}\label{NC Lorentz space norm I}
\|x\|_{L^{p,q}(\mathbb{R}^{d}_{\theta})}:=\left(\int_{0}^{+\infty}\left(t^\frac{1}{p}\mu(t,x)\right)^{q}\frac{dt}{t}\right)^{\frac1q}\,\ \text{for} \,\ q<\infty,
\end{equation}
and 
$$\|x\|_{L^{p,\infty}(\mathbb{R}^{d}_{\theta})}:=\sup_{t>0}t^{\frac{1}{p}}\mu(t,x) \,\ \text{for} \,\ q=\infty.$$
In other words,
\begin{equation}\label{mu norm}
\|x\|_{L^{p,q}(\mathbb{R}^{d}_{\theta})}=\|\mu(x)\|_{L^{p,q}(\mathbb{R}_+)}.
\end{equation}
In particular,
$L^{p,p}(\mathbb{R}^{d}_{\theta})=L^{p}(\mathbb{R}^{d}_{\theta})$ isometrically for $1\leq p\leq\infty$ with equivalent norms.
\end{definition}
These are noncommutative quasi-Banach spaces (resp. Banach for $1\leq p\leq \infty,$ $1\leq q<\infty$). For the theory of $L^p$ and Lorentz spaces corresponding to general semifinite von Neumann algebras, we refer the reader to  \cite{DPS}, \cite{LSZ}, \cite{PXu}.

\subsection{Differential calculus on $L^{\infty}(\mathbb{R}^{d}_{\theta})$} 
Now let us recall the differential structure on $\mathbb{R}^d_{\theta}$ (see, \cite[Subsection 2, p. 10]{Mc}).

It  is based on the group of translations $\{T_s\}_{s\in\mathbb{R}^d},$ where $T_s$ is defined as the unique $\ast$-automorphism of $L^{\infty}(\mathbb{R}^d_\theta)$ which acts on $U_{\theta}(t)$ as
\begin{equation}\label{def-translations}
T_s(U_{\theta}(t))=e^{i(t,s)}U_{\theta}(t), \quad  
t,s \in\mathbb{R}^d, 
\end{equation}
where $(\cdot,\cdot)$ means the usual inner product in $\mathbb{R}^d.$ 

Equivalently, for $x \in L^{\infty}(\mathbb{R}^d_\theta) \subseteq B(L^2(\mathbb{R}^d_\theta))$ we may define $T_s(x)$ as the conjugation of $x$ by the
unitary operator of translation by $s$  on $L^2(\mathbb{R}^d_\theta).$   
\begin{definition} (see, \cite[Definition 2.9]{Mc}) An element $x \in L^{1}(\mathbb{R}^d_\theta) + L^{\infty}(\mathbb{R}^d_\theta)$   is said to be smooth if for all $y \in L^{1}(\mathbb{R}^d_\theta) \cap L^{\infty}(\mathbb{R}^d_\theta)$ the function $s \mapsto \tau_{\theta}(yT_s(x))$ is smooth.
\end{definition}

The partial derivations $\partial^{\theta}_j,  j = 1, \dots, d,$ are defined on smooth elements $x$ by 
$$
\partial^{\theta}_j(x) = \frac{d}{ds_{j}}T_s(x)|_{s=0}. 
$$
From \eqref{def-translations} and \eqref{def-integration} it is easily verified that 
$$
\partial^{\theta}_j(x)\overset{\eqref{def-integration}}{=}\partial^{\theta}_j\lambda_{\theta}(f)\overset{\eqref{def-translations}}{=}\lambda_{\theta}(it_{j}f(t)), \quad  j=1,\cdots,d, \,\  f\in \mathcal{S}(\mathbb{R}^d),
$$
for $x=\lambda_{\theta}(f).$

For a multi-index $\alpha=(\alpha_1,...,\alpha_d),$ we define 
$$
\partial^{\alpha}_{\theta}=(\partial^{\theta}_{1})^{\alpha_{1}}\dots(\partial^{\theta}_{d})^{\alpha_{d}}, 
$$
and the gradient $\nabla_{\theta}$ associated with $L^{\infty}(\mathbb{R}_{\theta}^d)$ is the operator
$$
\nabla_{\theta}=(\partial^{\theta}_{1},  \dots,  \partial^{\theta}_{d}).  
$$
Moreover, the Laplace operator $\Delta_{\theta}$ is defined as
\begin{equation}\label{laplacian}
\Delta_{\theta} = (\partial_1^{\theta})^2   + \cdots +(\partial_d^{\theta})^2, 
\end{equation}
where  $-\Delta_{\theta}$ is a positive operator  on $L^2(\mathbb{R}^{d}_\theta)$ (see  \cite {MSX} and \cite{Mc}).

By duality, we can extend the derivatives $\partial^{\alpha}_{\theta}$ to operators on $\mathcal{S}'(\mathbb{R}^d).$ With these generalised derivatives, we are able to introduce the Sobolev spaces $L^p_s (\mathbb{R}^d)$ associated to the non-commutative Euclidean space.  For this we will have need to refer to the operator $ J_{\theta}= (1-\Delta_{\theta})^{\frac{1}{2}}.$ As in the classical case, the operator $J_{\theta}$ will be called the Bessel potential.

\begin{definition} (see, \cite[Definition 3.12.]{Mc}) For $1 \leq  p <\infty$ and $s\in\mathbb{R}, $ define the Bessel potential Sobolev space $L_{s}^p(\mathbb{R}^{d}_\theta)$ as the subset of $x\in\mathcal{S}'(\mathbb{R}^d)$ such that $J_{\theta}^s  x \in L^p(\mathbb{R}^{d}_\theta),$ with the norm
$$
\|x\|_{L_{s}^p(\mathbb{R}^{d}_\theta)}=\|J_{\theta}^sx\|_{L^p(\mathbb{R}^{d}_\theta)}. 
$$
\end{definition}
For a positive integer $m$ and $1 \leq p \leq \infty,$ the space $W^{p,m} (\mathbb{R}^{d}_\theta)$ is the space of $x\in\mathcal{S}'(\mathbb{R}^d)$ such that every partial derivative of $x$ up to order $m$ is in $L^p(\mathbb{R}^{d}_\theta),$ equipped with the norm
$$
\|x\|_{W^{p,m} (\mathbb{R}^{d}_\theta)}=\sum\limits_{ |\alpha|\leq m} \|\partial^{\alpha}_{\theta}x\|_{L^p(\mathbb{R}^{d}_\theta)}. 
$$

The above definitions of Sobolev spaces on quantum Euclidean spaces and their main properties   were studied in \cite[Section 3]{MSX}.

\subsection{Fourier transform and Fourier multiplier on quantum  Euclidean spaces} 

\begin{definition}\label{F-transform}
For any $x \in \mathcal{S}(\mathbb{R}_{\theta}^d),$ we define the Fourier transform of $x$ as the map $\lambda_{\theta}^{-1}:\mathcal{S}(\mathbb{R}_{\theta}^d)\to \mathcal{S}(\mathbb{R}^d)$ by the formula
\begin{equation}\label{direct-F-transform}
\lambda_{\theta}^{-1}(x):=\widehat{x}, \quad \widehat{x}(s)=\tau_{\theta}(xU_{\theta}(s)^*), \,\ s\in \mathbb{R}^d.
\end{equation}
\end{definition}

We will use the following quantum Euclidean analogues of the Hausdorff-Young inequality for the Fourier and inverse Fourier transforms in \cite[Lemma 2.4]{HLW} and \cite[Proposition 2.10]{MSX}, respectively.
\begin{lem}\cite[Lemma 2.4]{HLW}  Let $1\leq p\leq 2$ with $\frac{1}{p}+\frac{1}{p'}=1.$ Then the Fourier transform $\lambda_{\theta}^{-1}$ can be extended to a linear contraction from $L^{p}(\mathbb{R}^{d}_{\theta})$ to $L^{p'}(\mathbb{R}^{d})$ and we have \begin{equation}\label{direct-H-Y_ineq}
\|\widehat{x}\|_{L^{p'}(\mathbb{R}^{d})}\leq \|x\|_{L^{p}(\mathbb{R}^{d}_{\theta})}, \quad x\in \mathcal{S}(\mathbb{R}_{\theta}^d).
\end{equation}
  Moreover, 
  it becomes equality when $p=2,$ giving the Parseval (or Plancherel) identity
  \begin{equation}\label{Parseval}
\|\widehat{x}\|_{L^{2}(\mathbb{R}^{d})}= \|x\|_{L^{2}(\mathbb{R}^{d}_{\theta})}.
\end{equation}
\end{lem}

\begin{lem}\cite[Proposition 2.10]{MSX} Let $1\leq p\leq 2$ with $\frac{1}{p}+\frac{1}{p'}=1.$ Then $\lambda_{\theta}$ has a continuous extension from $L^{p}(\mathbb{R}^{d})$ to $L^{p'}(\mathbb{R}^{d}_{\theta}),$ and we have 
\begin{equation}\label{H-Y_ineq}
\|\lambda_{\theta}(f)\|_{L^{p'}(\mathbb{R}^{d}_{\theta})}\leq \|f\|_{L^{p}(\mathbb{R}^{d})},\quad f\in L^{1}(\mathbb{R}^{d})\cap  L^{p}(\mathbb{R}^{d}).
\end{equation}
\end{lem}
In particular, we have Plancherel's (Parseval's) identity
\begin{equation}\label{Plancherel}
\|\lambda_{\theta}(f)\|_{L^{2}(\mathbb{R}^{d}_{\theta})}=\|f\|_{L^{2}(\mathbb{R}^{d})},\quad f\in L^{2}(\mathbb{R}^{d}).
\end{equation}
Throughout this paper, we call $\lambda_{\theta}^{-1}(\cdot)$ and $\lambda_{\theta}(\cdot)$ as the noncommutative Fourier and inverse Fourier transforms on the quantum Euclidean space.
\begin{rem}\label{Remark} We have  the following relations
\begin{itemize}
\item[(i)] $\lambda_{\theta}(\widehat{x})=x,\quad x\in L^{2}(\mathbb{R}_{\theta}^{d})$;
\item[(ii)]  $\widehat{\lambda_{\theta}(f)}=f, \quad f\in L^2(\mathbb{R}^{d})$.
\end{itemize}    
\end{rem}

\begin{definition}\label{F-multiplier}Let $g:\mathbb{R}^d\to \mathbb{C}$ be a measurable function from Schwartz class. Then we define the noncommutative Fourier multiplier $g(D)$ on $\mathcal{S}(\mathbb{R}_{\theta}^d)$ by the formula 
\begin{equation}\label{Fourier-multiplier}
g(D)(x)=\int_{\mathbb{R}^d}g(s)\widehat{x}(s)U_{\theta}(s)ds.
\end{equation}
    
\end{definition}

Similarly, we can extend the definition of $g(D)$ to 
$\mathcal{S}'(\mathbb{R}_{\theta}^d)$ by the relation
$$(g(D)(F),x)=(F,g(D)(x)), \quad x\in \mathcal{S}(\mathbb{R}_{\theta}^d), F\in \mathcal{S}'(\mathbb{R}_{\theta}^d).$$

\begin{definition}\label{def-adj}
Let $1 \leq p,q \leq \infty$.  For a bounded linear operator $A:L^p(\mathbb{R}^d_\theta)\to L^q(\mathbb{R}^d_\theta)$ we denote by $A^*:L^{q'}(\mathbb{R}^d_\theta)\to L^{p'}(\mathbb{R}^d_\theta)$  its adjoint operator defined by
\begin{eqnarray}\label{def-dul-operator}
 (A(x),y):=\tau_{\theta}(A(x) y^*)=\tau_{\theta}(x(A^*(y))^*)=(x,A^*(y)),    
 \end{eqnarray} 
for all $x\in L^p(\mathbb{R}^d_\theta)$ and $y\in L^{q'}(\mathbb{R}^d_\theta)$ (or in a dense subspace of it).  
\end{definition}

We write $\mathcal{A}\lesssim \mathcal{B}$ if there is a constant $c> 0$ depending only on parameters of spaces such that
$\mathcal{A}\leq c\mathcal{B}.$ We write $\mathcal{A}\asymp \mathcal{B}$ if both $\mathcal{A}\lesssim \mathcal{B}$ and $\mathcal{A}\gtrsim \mathcal{B}$ hold, possibly with different
constants.

\section{Hausdorff–Young–Paley Inequality on the quantum Euclidean space}\label{sc4}

In this section, we give an analogue of Hausdorff–Young–Paley inequality on the quantum Euclidean space. 
The following Plancherel (or Parseval) relation is necessary. 
\begin{lem} \label{P-relation}For any $x,y\in \mathcal{S}(\mathbb{R}^d_{\theta}),$ we have the equality
 \begin{eqnarray}\label{Plancherel-relation}
  \tau_{\theta}(xy^*)=\int_{\mathbb{R}^d}\widehat{x}(s)\overline{\widehat{y}(s)}ds. 
 \end{eqnarray}
\end{lem}
\begin{proof}   Let $x\in \mathcal{S}(\mathbb{R}^d_{\theta}).$ Then, by definition of $\mathcal{S}(\mathbb{R}^d_{\theta})$ there exists $f\in \mathcal{S}(\mathbb{R}^d)$ such that $x=\lambda_{\theta}(f).$ In other words, we have
\begin{equation}\label{*}
 x=\int_{\mathbb{R}^d}f(s)U_{\theta}(s)ds:= \lambda_{\theta}(f), \quad f\in  \mathcal{S}(\mathbb{R}^d).
\end{equation}
 It follows that
 \begin{eqnarray*} 
 \widehat{x}(s)&=&\tau_{\theta}(xU_{\theta}(s)^*)=\tau_{\theta}(xU_{\theta}(-s))\overset{\eqref{*}}{=}\tau_{\theta}\left(\int_{\mathbb{R}^d}f(\xi)U_{\theta}(\xi)U_{\theta}(-s)d\xi\right)\nonumber\\
 &\overset{\eqref{weyl-relation}}{=}&\tau_{\theta}\left(\int_{\mathbb{R}^d}f(\xi)e^{-\frac{i}{2}(\xi,\theta s)}U_{\theta}(\xi-s)d\xi\right)\nonumber\\
 &\overset{\xi-s=\eta}{=}&\tau_{\theta}\left(\int_{\mathbb{R}^d}f(\eta +s)e^{-\frac{i}{2}(\eta+s,\theta s)}U_{\theta}(\eta)d\eta\right)\nonumber\\ &=&\tau_{\theta}\left(\int_{\mathbb{R}^d}f(\eta +s)e^{-\frac{i}{2}(\eta,\theta s)}e^{-\frac{i}{2}(s,\theta s)}U_{\theta}(\eta)d\eta\right)\nonumber\\
 &=&\tau_{\theta}\left(\int_{\mathbb{R}^d}f(\eta +s)e^{-\frac{i}{2}(\eta,\theta s)}U_{\theta}(\eta)d\eta\right)\nonumber\\
 &=&\tau_{\theta}\left(\lambda_{\theta}(f(\cdot +s)e^{-\frac{i}{2}(\cdot,\theta s)})\right)\nonumber\\
 &\overset{\eqref{trace-def}}{=}&f(0+s)e^{-\frac{i}{2}(0,\theta s)}=f(s).
 \end{eqnarray*} 
Here, in the fourth line we used the fact that $(s,\theta s)=0.$ Indeed, since $\theta$ is the anti-symmetric matrix, it follows that $(\theta s,t)=-(s,\theta t)$ for any $s,t\in \mathbb{R}^d.$ In particular, $(\theta s,s)+(s,\theta s)=0, s\in \mathbb{R}^d,$ which implies $2(s,\theta s)=0, s\in \mathbb{R}^d.$ Consequently, $(s,\theta s)=0.$
 So, we have 
 \begin{equation}\label{**}
 \widehat{x}(s)=f(s), \quad f\in  \mathcal{S}(\mathbb{R}^d),  \,\ s\in\mathbb{R}^d. 
\end{equation}
By using this, we obtain 
$$\int_{\mathbb{R}^d}\widehat{x}(s)U_{\theta}(s)ds\overset{\eqref{**}}{=}\int_{\mathbb{R}^d}f(s)U_{\theta}(s)ds=\lambda_{\theta}(f)=x.$$
In other words,
\begin{eqnarray}\label{fourier transform}
x=\int_{\mathbb{R}^d}\widehat{x}(s)U_{\theta}(s)ds, \quad x\in\mathcal{S}(\mathbb{R}_{\theta}^d).
\end{eqnarray}
Therefore, 
\begin{eqnarray}\label{conjugate-fourier transform}
x^*\overset{\eqref{**}}{=}\left(\int_{\mathbb{R}^d} \widehat{x}(s)U_{\theta}(s)ds\right)^*=\int_{\mathbb{R}^d}\overline{\widehat{x}(s)}U^*_{\theta}(s)ds, \quad x\in\mathcal{S}(\mathbb{R}_{\theta}^d).
\end{eqnarray}
Thus, we have
\begin{eqnarray*} 
\tau_{\theta}(xy^*)&\overset{\eqref{conjugate-fourier transform}}{=}&\tau_{\theta}\left(\int\limits_{\mathbb{R}^d}f(s)U_{\theta}(s)ds \int\limits_{\mathbb{R}^d}\overline{\widehat{y}(t)} U^*_{\theta}(t)dt\right)\\
&=&\tau_{\theta}\left(\int\limits_{\mathbb{R}^d}\left(\int\limits_{\mathbb{R}^d}f(s)U_{\theta}(s)U_{\theta}(-t)ds\right)\overline{\widehat{y}(t)}dt\right)\\
&=&\int\limits_{\mathbb{R}^d}\tau_{\theta}\left(\int\limits_{\mathbb{R}^d}f(s)U_{\theta}(s) U_{\theta}(-t)ds \right)\overline{\widehat{y}(t)}dt \\
&\overset{\eqref{fourier transform}}{=}&\int\limits_{\mathbb{R}^d}\widehat{x}(t)\overline{\widehat{y}(t)}dt,\quad  \forall x,y\in \mathcal{S}(\mathbb{R}^d_{\theta}),
 \end{eqnarray*}  
which is the desired result. 
\end{proof}

Next, we prove that inequality \eqref{direct-H-Y_ineq} remains true with the reversed inequality for $2\leq p\leq \infty.$

\begin{prop}\label{R-H-Y_ineq} If $2\leq p\leq\infty$ and if $\widehat{x}\in L^{p'}(\mathbb{R}^d),$ then the operator $x$ defined by $\widehat{x}$ belongs to $L^{p}(\mathbb{R}^d_\theta),$ and we have
 \begin{eqnarray}\label{R_H-Y_ineq}
\|x\|_{L^p(\mathbb{R}^d_\theta)} \leq \|\widehat{x}\|_{L^{p'}(\mathbb{R}^d )}, 
\end{eqnarray}
where $\frac{1}{p}+\frac{1}{p'}=1.$
 \end{prop}
\begin{proof}By duality of $L^p(\mathbb{R}^{d}_{\theta})$ we have
  $$\|x\|_{L^p(\mathbb{R}^{d}_{\theta})}=\sup\left\{\left|\tau_{\theta}(xy^*)\right|:y\in L^{p'}(\mathbb{R}^{d}_{\theta}),\quad
  \|y\|_{L^{p'}(\mathbb{R}^{d}_{\theta})}=1\right\}.$$
  By Lemma \ref{P-relation} we have
  $$
  \tau_{\theta}(xy^*)=\int_{\mathbb{R}^d}\widehat{x}(s)\overline{\widehat{y}(s)}ds, \quad x,y\in \mathcal{S}(\mathbb{R}^d_{\theta}).
  $$
  Since $|\widehat{x}(s)\overline{\widehat{y}(s)}|\leq |\widehat{x}(s)||\overline{\widehat{y}(s)}|$ for any $s\in\mathbb{R}^d,$ applying the classical H\"{o}lder inequality with respect to Fourier transforms of $x$ and $y\in L^{p'}(\mathbb{R}^{d}_{\theta})$ with $\|y\|_{L^{p'}(\mathbb{R}^{d}_{\theta})}=1,$ we obtain
  \begin{eqnarray*}
\|x\|_{L^p(\mathbb{R}^d_\theta)}&=&\sup\{|\tau_{\theta}(xy^*)|:y\in L^{p'}(\mathbb{R}^{d}_{\theta}),\quad
  \|y\|_{L^{p'}(\mathbb{R}^{d}_{\theta})}=1\}\\
&=&\sup\left\{\left|\int_{\mathbb{R}^d}\widehat{x}(s)\overline{\widehat{y}(s)}ds\right|:y\in L^{p'}(\mathbb{R}^{d}_{\theta}),\quad
  \|y\|_{L^{p'}(\mathbb{R}^{d}_{\theta})}=1\right\}\\
&\leq &\sup\left\{\int_{\mathbb{R}^d}|\widehat{x}(s)\overline{\widehat{y}(s)}|ds:y\in L^{p'}(\mathbb{R}^{d}_{\theta}),\quad
  \|y\|_{L^{p'}(\mathbb{R}^{d}_{\theta})}=1\right\}\\
&\leq &\sup\left\{\int_{\mathbb{R}^d}|\widehat{x}(s)||\widehat{y}(s)|ds:y\in L^{p'}(\mathbb{R}^{d}_{\theta}),\quad
  \|y\|_{L^{p'}(\mathbb{R}^{d}_{\theta})}=1\right\} \\
    &\leq&\sup_{y\in L^{p'}(\mathbb{R}^{d}_{\theta})\atop
  \|y\|_{L^{p'}(\mathbb{R}^{d}_{\theta})}=1}\left
  \{\left(\int_{\mathbb{R}^d}|\widehat{x}(s)|^{p'}ds\right)^{1/p'}\cdot\left(\int_{\mathbb{R}^d}|\widehat{y}(s)|^{p}ds\right)^{1/p}\right\}\\
    &=&\sup_{y\in L^{p'}(\mathbb{R}^{d}_{\theta})\atop
  \|y\|_{L^{p'}(\mathbb{R}^{d}_{\theta})}=1}\left\{ \|\widehat{x}\|_{L^{p'}(\mathbb{R}^{d})}\cdot \|\widehat{y}\|_{L^{p}(\mathbb{R}^{d})}
 \right\}.
\end{eqnarray*} 
Now applying the Hausdorff-Young inequality \eqref{direct-H-Y_ineq} with respect to $y,$ we get
\begin{eqnarray*}
\|x\|_{L^p(\mathbb{R}^d_\theta)}&\leq&\sup_{y\in L^{p'}(\mathbb{R}^{d}_{\theta})\atop
  \|y\|_{L^{p'}(\mathbb{R}^{d}_{\theta})}=1}\left\{ \|\widehat{x}\|_{L^{p'}(\mathbb{R}^{d})}\cdot \|\widehat{y}\|_{L^{p}(\mathbb{R}^{d})} 
 \right\}\\
 &\overset{\eqref{direct-H-Y_ineq}}{\leq}& \sup_{y\in L^{p'}(\mathbb{R}^{d}_{\theta})\atop
  \|y\|_{L^{p'}(\mathbb{R}^{d}_{\theta})}=1}\left\{ \|\widehat{x}\|_{L^{p'}(\mathbb{R}^{d})}\cdot \|y\|_{L^{p'}(\mathbb{R}^{d}_{\theta})} 
 \right\}= \|\widehat{x}\|_{L^{p'}(\mathbb{R}^{d})},
\end{eqnarray*} 
thereby completing the proof.
\end{proof}

In order to prove the Hausdorff–Young–Paley inequality, we first need to prove the following Paley-type inequality.
\begin{thm}\label{Direct_P_ineq} (Paley-type inequality). Let $1 < p \leq 2$ and let $h:\mathbb{R}^{d}\to \mathbb{R}_{+}$ be a strictly positive function such that   
 \begin{eqnarray}\label{Cond_D_P_ineq}
 M_h:=\sup\limits_{t>0}t\int\limits_{h(s)\geq t}ds<\infty.
 \end{eqnarray}
Then for any $x\in L^p(\mathbb{R}_{\theta}^d)$ we have 
\begin{eqnarray}\label{D_P_ineq}
\left(\int\limits_{\mathbb{R}^d} |\widehat{x}(s)|^p h^{2-p}(s)ds \right)^\frac{1}{p}\leq c_p M^\frac{2-p}{p}_h\|x\|_{L^p(\mathbb{R}_{\theta}^d)},
\end{eqnarray}  
where $c_p>0$ is a constant independent of $x.$ 
  \end{thm}
\begin{proof}Let $\nu$ be a measure on $\mathbb{R}^{d}$ such that 
$\nu(s):= h^2(s)ds>0.$ Define the space $L^p(\mathbb{R}^{d}, \nu)$ as the space of all complex valued measurable functions $f$ on $\mathbb{R}^{d}$ such that
\begin{eqnarray*} 
\|f\|_{L^p(\mathbb{R}^{d}, \nu)}:=\left(\int\limits_{\mathbb{R}^{d}}|f(s)|^p h^2(s)ds\right)^\frac{1}{p}<\infty.
\end{eqnarray*}
It is not difficult to see that this is a Banach space with the above norm.
Now, define an operator $A:L^p(\mathbb{R}_{\theta}^{d})\to L^p(\mathbb{R}^{d}, \nu)$ by formula 
\begin{eqnarray*} 
A(x)=\frac{\widehat{x}}{h}.
\end{eqnarray*}
Since 
$$\widehat{x+y}(s)=\tau_{\theta}((x+y)U_{\theta}(s)^*)=\tau_{\theta}(xU_{\theta}(s)^*)+\tau_{\theta}(yU_{\theta}(s)^*)=\widehat{x}(s)+\widehat{y}(s), \quad x,y\in \mathcal{S}(\mathbb{R}_{\theta}^d),$$ it follows that $A$ is a linear operator. Next, we will show that $A$
is well-defined and bounded from $L^p(\mathbb{R}_{\theta}^d)$ to  $L^p(\mathbb{R}^d, \nu)$ for $1 \leq p \leq 2$. In other words, we claim that we have the estimate \eqref{D_P_ineq}
 with the condition \eqref{Cond_D_P_ineq}. First, we will show that $A$ is of a weak type $(2,2)$ and of a weak-type $(1,1)$. Recall that the distribution function $d_{A(f)}(t),$ $t>0,$ with respect
to the weight $h^2(s)>0$ is defined by
 \begin{eqnarray*} 
d_{A(x)}(t):=\nu\{s>0: |A(x)|>t\}=\int\limits_{|A(x)|>t}h^2(s)ds. 
\end{eqnarray*} 
Now, we show that 
\begin{eqnarray}\label{D-additive4.3}
d_{A(x)}(t)\leq\left(\frac{c_2\|x\|_{L^2(\mathbb{R}^d_\theta)}}{t}\right)^2  \;\;\;\text{with}\;\;\;c_2=1,
\end{eqnarray}
 and 
\begin{eqnarray}\label{D-additive4.4}
d_{A(x)}(t)\leq \frac{c_1\|x\|_{L^1(\mathbb{R}^d_\theta)}}{t}   \;\;\;\text{with}\;\;\;c_1=2M_h. 
\end{eqnarray}

Indeed,  we first prove  the inequality  (\ref{D-additive4.3}). It follows from the Chebyshev's inequality and Plancherel's identity (\ref{Parseval}) that
 \begin{eqnarray*} 
t^2d_{A(x)}(t)\leq \|A(x)\|^2_{L^2(\mathbb{R}^d, \nu)}
= \int\limits_{\mathbb{R}^d} |\widehat{x}(s)|^2 ds=\|\widehat{x}\|^2_{L^{2}(\mathbb{R}^d)}
\overset{ \text{(\ref{Parseval})}}{=}\|x\|^2_{L^2(\mathbb{R}^d_{\theta})}.
\end{eqnarray*} 
Thus, $A$ is of weak type $(2,2)$ with operator norm at most $c_2=1$.  
Now, let us show \eqref{D-additive4.4}. Applying the noncommutative H\"{o}lder inequality with respect to $p=1$ and $p'=\infty,$ we have
\begin{eqnarray*}
\frac{|\widehat{x}(s)|}{h(s)} \leq \frac{|\tau_{\theta}(xU_{\theta}(s)^*)|}{h(s)} \leq\frac{\|U_{\theta}(s)\|_{L^\infty(\mathbb{R}^d_\theta)}\|x\|_{L^1(\mathbb{R}_{\theta}^d)}}{h(s)}\leq\frac{\|x\|_{L^1(\mathbb{R}^d_{\theta})}}{h(s)}, \quad s\in\mathbb{R}^d.
\end{eqnarray*} 
Therefore, we have
\begin{eqnarray*}
 \{s\in\mathbb{R}^d: \frac{|\widehat{x}(s)|}{h(s)}>t\} \subset
 \{s\in\mathbb{R}^d: \frac{\|x\|_{L^1(\mathbb{R}_{\theta}^d)} }{h(s)}>t\}
\end{eqnarray*} for any $t>0.$
Consequently,
  \begin{eqnarray*}
 \nu \{s\in\mathbb{R}^d: \frac{|\widehat{x}(s)|}{h(s)}>t\} \leq
 \nu\{s\in\mathbb{R}^d: \frac{\|x\|_{L^1(\mathbb{R}_{\theta}^d)} }{h(s)}>t\}
\end{eqnarray*}
for any $t>0.$
Setting  $v:=\frac{\|x\|_{L^1(\mathbb{R}_{\theta}^d)}}{t}$, we obtain 
\begin{equation}\label{D-additive4.5}
  \nu \{s\in\mathbb{R}^d: \frac{|\widehat{x}(s)|}{h(s)}>t\} \leq
 \nu\{s\in\mathbb{R}^d: \frac{\|x\|_{L^1(\mathbb{R}_{\theta}^d)} }{h(s)}>t\}=\int\limits_{h(s)\leq v}h^2(s)ds.
\end{equation}
Let us estimate the right hand side. We now claim that
\begin{eqnarray}\label{D-additive4.6}
\int\limits_{h(s)\leq v}h^2(s)ds\leq 2v\cdot M_h .
\end{eqnarray}
Indeed, first we have
$$\int\limits_{h(s)\leq v}h^2(s)ds=\int\limits_{  h(s)\leq v}\int\limits_0^{h^2(s)}d\xi.$$
By interchanging the order of integration we obtain
$$\int\limits_{h(s)\leq v}ds\int\limits_0^{h^2(s)}d\xi=\int\limits_0^{v^2}d\xi\int\limits_{\xi^\frac{1}{2}\leq h(s)\leq v}ds.$$
Further, by a substitution $\xi = t^2$ we have
$$\int\limits_0^{v^2}d\xi\int\limits_{\xi^\frac{1}{2}\leq h(s)\leq v}ds=2\int\limits_0^{v}tdt\int\limits_{t\leq h(s)\leq v}ds\leq 2\int\limits_0^{v}t dt\int\limits_{t\leq h(s)}ds.$$
Since 
$$t \int\limits_{t\leq h(s)}ds\leq \sup_{t>0}t\int\limits_{t \leq h(s)}ds= M_{h}$$
and $M_{h}<\infty$ by assumption, it follows that
$$2\int\limits_0^{v}t dt\int\limits_{t\leq h(s)}ds\leq 2 v\cdot M_{h}.$$
This proves claim \eqref{D-additive4.6}.
Combining \eqref{D-additive4.5} and \eqref{D-additive4.6} we obtain (\ref{D-additive4.4}) which shows that $A$ is indeed of weak type $(1,1)$ with operator norm at most $c_1 = 2M_h$. Then, by using Marcinkiewicz interpolation theorem (see, \cite[Theorem 1.3.1]{BL1976}) with $p_1 = 1,$ $p_2 = 2$
and $\frac{1}{p}=\frac{1-\eta}{1}+\frac{\eta}{2},$ we now obtain the inequality (\ref{D_P_ineq}).  This completes the proof.  
\end{proof}

\begin{rem}Let $\theta=0$ and $d=1.$ If $h(s)=\frac{1}{1+|s|}, s\in\mathbb{R},$ then $M_h<\infty$ in Theorem \ref{Direct_P_ineq}. Consequently, we obtain the classical Hardy-Littlewood inequality \cite[Theorem 2.1]{ANR}), \cite{HL} (see also \cite{Y} for compact quantum groups of Kac type).
    
\end{rem}
As a consequence of the Paley-type inequality in Theorem \ref{Direct_P_ineq}, we obtain the following Hardy-Littlewood inequality on the quantum Euclidean space. 

\begin{thm}\label{Direct-H-L_ineq} (Hardy-Littlewood type inequality). Let $1 < p \leq 2$ and let $\varphi:\mathbb{R}^d\to \mathbb{R}_{+}$ be a strictly positive
function such that 
$$\int\limits_{\mathbb{R}^d}\frac{1}{\varphi^{\beta}(s)}ds<\infty \quad \text{for}\quad \text{some} \quad \beta>0.
$$
 Then for any $x\in L^p(\mathbb{R}^d_{\theta})$ we have 
\begin{eqnarray}\label{D-H_L_ineq}
\left(\int\limits_{\mathbb{R}^{d}} |\widehat{x}(s)|^p \varphi^{\beta(p-2)}(s)ds \right)^\frac{1}{p}\leq c_p \|x\|_{L^p(\mathbb{R}^d_{\theta})},
 \end{eqnarray}   
 where $c_p >0$ is a constant independent of $x.$
  \end{thm}
  \begin{proof}By the assumption, we have
  $$C:=\int\limits_{\mathbb{R}^d}\frac{1}{\varphi^{\beta}(s)}ds<\infty\quad \text{for}\quad \text{some} \quad \beta> 0.$$
Thus,
$$C\geq \int\limits_{\varphi^{\beta}(s)\leq \frac{1}{t}}\frac{1}{\varphi^{\beta}(s)}ds\geq t\int\limits_{\varphi^{\beta}(s)\leq \frac{1}{t}}ds=t\int\limits_{t\leq \frac{1}{\varphi^{\beta}(s)}}ds,$$
where $t>0.$
Hence, taking supremum with respect to $t>0,$ we obtain
$$\sup_{t>0}t\int\limits_{t\leq \frac{1}{\varphi^{\beta}(s)}}ds\leq C<\infty.$$
Then, applying Theorem \ref{Direct_P_ineq} for the function $h(s)=\frac{1}{\varphi^{\beta}(s)},$ $s\in\mathbb{R}^d,$
we obtain the desired result.
  \end{proof}

\begin{thm}\label{R-H-L_ineq} (Inverse Hardy-Littlewood type inequality). Let $2 \leq p <\infty$  with $\frac{1}{p}+\frac{1}{p'}=1$ and let $\varphi:\mathbb{R}^d\to \mathbb{R}_{+}$ be a strictly positive
function such that 
$$\int\limits_{\mathbb{R}^d}\frac{1}{\varphi^{\beta}(s)}ds<\infty \quad \text{for}\quad \text{some} \quad \beta> 0.
$$ 
If $$\int\limits_{\mathbb{R}^d}|\widehat{x}(s)|^p\varphi^{\frac{\beta p(2-p')}{p'}}(s)ds<\infty,
$$ then $x\in L^p(\mathbb{R}^{d}_{\theta})$ and we have
\begin{eqnarray}\label{R_H_L_ineq}
\|x\|^p_{L^p(\mathbb{R}^d_\theta)}\leq C_p\int\limits_{\mathbb{R}^d}|\widehat{x}(s)|^p\varphi^{\frac{\beta p(2-p')}{p'}}(s)ds,
 \end{eqnarray}   
 where $C_p>0$ is a constant independent of $x.$
  \end{thm}
  \begin{proof}By duality of $L_p(\mathbb{R}^{d}_{\theta})$ we have
  $$\|x\|_{L^p(\mathbb{R}^{d}_{\theta})}=\sup\left\{\left|\tau_{\theta}(xy^*)\right|:y\in L^{p'}(\mathbb{R}^{d}_{\theta}),\quad
  \|y\|_{L^{p'}(\mathbb{R}^{d}_{\theta})}=1\right\}.$$
  By Lemma \ref{P-relation} we have 
 \begin{eqnarray*} 
  \tau_{\theta}(xy^*)=\int_{\mathbb{R}^d}\widehat{x}(s)\overline{\widehat{y}(s)}ds, \quad x,y\in \mathcal{S}(\mathbb{R}^d_{\theta}).
 \end{eqnarray*}
  Since $|\widehat{x}(s)\overline{\widehat{y}(s)}|\leq |\widehat{x}(s)||\overline{\widehat{y}(s)}|$ for any $s\in\mathbb{R}^d,$ applying the classical H\"{o}lder inequality for any $y\in L^{p'}(\mathbb{R}^{d}_{\theta})$ with $\|y\|_{L^{p'}(\mathbb{R}^{d}_{\theta})}=1,$ we obtain
  \begin{eqnarray*}
\|x\|_{L^p(\mathbb{T}^d_\theta)}&=&\sup\{|\tau_{\theta}(xy^*)|:y\in L^{p'}(\mathbb{R}^{d}_{\theta}),\quad
  \|y\|_{L^{p'}(\mathbb{R}^{d}_{\theta})}=1\}\\
&=&\sup\left\{\left|\int_{\mathbb{R}^d}\widehat{x}(s)\overline{\widehat{y}(s)}ds\right|:y\in L^{p'}(\mathbb{R}^{d}_{\theta}),\quad
  \|y\|_{L^{p'}(\mathbb{R}^{d}_{\theta})}=1\right\}\\
&\leq &\sup\left\{\int_{\mathbb{R}^d}|\widehat{x}(s)\overline{\widehat{y}(s)}|ds:y\in L^{p'}(\mathbb{R}^{d}_{\theta}),\quad
  \|y\|_{L^{p'}(\mathbb{R}^{d}_{\theta})}=1\right\}\\
&\leq &\sup\left\{\int_{\mathbb{R}^d}|\widehat{x}(s)||\widehat{y}(s)|ds:y\in L^{p'}(\mathbb{R}^{d}_{\theta}),\quad
  \|y\|_{L^{p'}(\mathbb{R}^{d}_{\theta})}=1\right\} \\
  &\leq&\sup\left\{\int_{\mathbb{R}^d}\varphi^{\frac{\beta (2-p')}{p'}}(s)|\widehat{x}(s)|\cdot\varphi^{\frac{\beta (p'-2)}{p'}}(s)|\widehat{y}(s)|ds:y\in L^{p'}(\mathbb{R}^{d}_{\theta}),\quad
  \|y\|_{L^{p'}(\mathbb{R}^{d}_{\theta})}=1\right\} \\
  &\leq&\sup_{y\in L^{p'}(\mathbb{R}^{d}_{\theta})\atop
  \|y\|_{L^{p'}(\mathbb{R}^{d}_{\theta})}=1}\left\{\left(\int_{\mathbb{R}^d}\varphi^{\frac{\beta p(2-p')}{p'}}(s)|\widehat{x}(s)|^pds\right)^{1/p}\cdot\left(\int_{\mathbb{R}^d}\varphi^{\beta (p'-2)}(s)|\widehat{y}(s)|^{p'}ds\right)^{1/p'}
 \right\}.
\end{eqnarray*} 
Now applying Theorem \ref{Direct-H-L_ineq} with respect to $p',$ we get 
 \begin{eqnarray*}
\|x\|_{L^p(\mathbb{R}^d_\theta)}
  &\leq&\sup_{y\in L^{p'}(\mathbb{R}^{d}_{\theta})\atop
  \|y\|_{L^{p'}(\mathbb{R}^{d}_{\theta})}=1}\left\{\left(\int_{\mathbb{R}^d}\varphi^{\frac{\beta p(2-p')}{p'}}(s)|\widehat{x}(s)|^p\right)^{1/p}\cdot\left(\int_{\mathbb{R}^d}\varphi^{\beta (p'-2)}(s)|\widehat{y}(s)|^{p'}\right)^{1/p'}
 \right\}\\
 &\lesssim&\left(\int_{\mathbb{R}^d}\varphi^{\frac{\beta p(2-p')}{p'}}(s)|\widehat{x}(s)|^p\right)^{1/p}\cdot\sup_{y\in L^{p'}(\mathbb{R}^{d}_{\theta})\atop
  \|y\|_{L^{p'}(\mathbb{R}^{d}_{\theta})}=1}\|y\|_{L^{p'}(\mathbb{R}^{d}_{\theta})}.
 \end{eqnarray*} 
Since $\|y\|_{L^{p'}(\mathbb{R}^{d}_{\theta})}=1,$ taking $C_p=c_{p'},$ we complete the proof.
  \end{proof}
\begin{rem} If $p=2,$ then the statements in Theorems \ref{Direct_P_ineq}, \ref{Direct-H-L_ineq}, and \ref{R-H-L_ineq} reduce to the Plancherel identity \eqref{Parseval}.    
\end{rem}
The following result can be inferred from \cite[Corollary 5.5.2, p. 120]{BL1976}.
\begin{prop}\label{weithg_inter}
  Let $ d\nu_1(\xi) = \omega_1(\xi)d\xi$, $d\nu_2(\xi)=\omega_2(\xi)d\xi, \,\ \xi\in\mathbb{R}^{d}.$ Suppose that $1\leq p,q_0, q_1 < \infty.$ If a continuous linear operator $A$ admits bounded extensions $A:L^p(\mathbb{R}^{d}_{\theta}) \rightarrow
L^{q_0} (\mathbb{R}^d,\nu_1)$  and $A:L^p(\mathbb{R}^{d}_{\theta}) \rightarrow L^{q_1}(\mathbb{R}^d,\nu_2),$ then there exists a bounded extension $A:L^p(\mathbb{R}^{d}_{\theta})\rightarrow  L^q(\mathbb{R}^d,\nu)$ where $0 < \eta < 1,$ $\frac{1}{q}=\frac{1-\eta }{q_0}+\frac{\eta}{q_1}$ and $d\nu=\omega(\xi)d\xi, \,\ \omega=\omega_1^\frac{q(1-\eta)}{q_0}\cdot\omega_2^\frac{q\eta }{q_1}$.
\end{prop}

\begin{thm}\label{D_H-Y-P_ineq}(Hausdorff–Young–Paley inequality). Let $1 < p \leq r \leq p' < \infty$ with $\frac{1}{p}+\frac{1}{p'}=1.$ Let $h$ is given as in Theorem \ref{Direct_P_ineq}. Then we have
\begin{eqnarray}\label{D-additive4.7}
\left(\int\limits_{\mathbb{R}^d} |\widehat{x}(\xi)|^r h(\xi)^{r\Big(\frac{1}{r}-\frac{1}{p'}\Big)} d\xi\right)^\frac{1}{r}\leq c_{p,r,p'} M^{\frac{1}{r}-\frac{1}{p'}}_h\|x\|_{L^p(\mathbb{R}^d_\theta)},
 \end{eqnarray}  
 where $c_{p,r,p'}>0$ is a constant independent of $x.$
\end{thm}

\begin{proof}First, let us consider a linear operator $ A(x):=\widehat{x}$ on $L^p(\mathbb{R}_{\theta}^d).$ Hence, applying the Paley-type inequality \eqref{D_P_ineq} with $1<p\leq 2,$ we obtain
\begin{eqnarray}\label{D-additive4.8}
\left(\int\limits_{\mathbb{R}^d} |\widehat{x}(\xi)| h^{2-p}(\xi)d\xi \right)^\frac{1}{p}\lesssim M^\frac{2-p}{p}_h\|x\|_{L^p(\mathbb{R}_{\theta}^d)}.
 \end{eqnarray} 
 In other words, $A$ is bounded from $L^p(\mathbb{R}^d_{\theta})$ to the weighted space $L^{p}(\mathbb{R}^d,\nu_1)$
 with the weight $\omega_1(\xi):= h^{2-p}(\xi)>0$ for $\xi\in\mathbb{R}^d.$
On the other hand, by the Hausdorff-Young inequality \eqref{direct-H-Y_ineq}, we have
\begin{eqnarray*}
\left(\int_{\mathbb{R}^d}|\widehat{x}(\xi)|^{p'}d\xi\right)^{1/p'}=\|\widehat{x}\|_{L^{p'}(\mathbb{R}^d)}\leq \|x\|_{L^p(\mathbb{R}_{\theta}^d)},
\end{eqnarray*} 
where $1\leq{p}\leq2$ with $\frac{1}{p}+\frac{1}{p'}=1$.
This shows that $A$ is bounded from $L^p(\mathbb{R}_{\theta}^d)$ to $L^{p'}(\mathbb{R}^d,\nu_2),$
where $\nu_2(\xi):=1d\xi$ for all $\xi\in \mathbb{R}^d.$ By Proposition \ref{weithg_inter} we infer that $A:L^p(\mathbb{R}^d_{\theta})\rightarrow L^r(\mathbb{R}^d, \nu)$ with $d\nu=\omega(\xi)d\xi,$ is bounded for any $r$ such that $p \leq r \leq p',$  where the space $L^r(\mathbb{R}^d, \nu)$ is defined as the space of all Lebesgue measurable functions $f$ on $\mathbb{R}^d$ with the norm 
$$\|f\|_{L^r(\mathbb{R}^d, \nu)}:=\left(\int\limits_{\mathbb{R}^d} |f(\xi)|^r \omega(\xi)d\xi \right)^\frac{1}{r}$$
and $\omega$ is a positive measurable function on $\mathbb{R}^d$ to be determined. Let us compute $\omega$ in our setting. Indeed, fix $\eta\in (0,1)$ such that $\frac{1}{r}=\frac{1-\eta}{p}+\frac{\eta}{p'}.$ Then we have that $\eta=\frac{p-r}{r(p-2)}$ and by Proposition \ref{weithg_inter} with respect to $q=r,$ $q_0=p,$ and $q_1=p',$ we obtain 
\begin{eqnarray*}\label{weight}
\omega(\xi)=(\omega_{1}(\xi))^{\frac{r(1-\eta)}{p}}\cdot(\omega_{2}(\xi))^{\frac{r\eta}{p'}}=(h^{2-p}(\xi))^{\frac{r(1-\eta)}{p}}\cdot1^{\frac{r\eta}{p'}}=h^{1-\frac{r}{p'}}(\xi)= h^{r\big(\frac{1}{r}-\frac{1}{p'}\big)}(\xi)
\end{eqnarray*}
for all $\xi\in\mathbb{R}^d$ and $\frac{2-p}{p}\cdot (1-\eta)=\frac{1}{r}-\frac{1}{p'}.$
Thus, 
$$\|A(x)\|_{L^{r}(\mathbb{R}^d, \nu)}\lesssim (M^\frac{2-p}{p}_h)^{1-\eta}\|x\|_{L^p(\mathbb{R}^d_{\theta})}=  M^{\frac{1}{r}-\frac{1}{p'}}_h  \|x\|_{L^p(\mathbb{R}^d_{\theta})}, \quad x\in L^p(\mathbb{R}^d_{\theta}),$$
where $d\nu=h^{r\big(\frac{1}{r}-\frac{1}{p'}\big)}(\xi)d\xi.$
This completes the proof.
\end{proof}

\section{A H\"ormander multiplier theorem on the quantum Euclidean space}\label{sc5}

In this section, we are concerned with the question of what assumptions on the symbol $g:\mathbb{R}^d\to \mathbb{C}$ guarantee that $g(D)$  is bounded from $L^p (\mathbb{R}^d_\theta )$  to $L^q (\mathbb{R}^d_\theta ).$ 

We first need to prove the following auxiliary results.
\begin{lem}\label{F_trans_Fourier_trans}We have

\begin{equation}\label{F_trans_Fourier_trans}
\widehat{g(D)(x)}=g\cdot\widehat{x}, \quad g\in \mathcal{S}(\mathbb{R}^d), \,\ x\in \mathcal{S}(\mathbb{R}^d_{\theta}).
\end{equation}
\end{lem}
\begin{proof}
The mapping $\mathcal{S}(\mathbb{R}^d)\ni \widehat{x} \mapsto \lambda_{\theta}(\widehat{x})$ is continuous between $\mathcal{S}(\mathbb{R}^d)$ and $\mathcal{S}(\mathbb{R}_{\theta}^d),$ and it has an extension from $\mathcal{S}'(\mathbb{R}^d)$ to $\mathcal{S}'(\mathbb{R}_{\theta}^d).$
By Definition \ref{F-multiplier} and \eqref{**}, 
\begin{equation}\label{def_F_multip}
g(D)(\lambda_{\theta}(f)) =g(D)(x)=\int_{\mathbb{R}^d}g(s)\widehat{x}(s)U_{\theta}(s)ds 
 \overset{\eqref{**}}{=} \int_{\mathbb{R}^d}g(s)f(s)U_{\theta}(s)ds=\lambda_{\theta}(gf).
\end{equation}
Hence,
$$g(D)(x)=g(D)(\lambda_{\theta}(\widehat{x}))\overset{\eqref{def_F_multip}}{=}\lambda_{\theta}(g\widehat{x}).$$
On the other hand, we have $g(D)(x)=\lambda_{\theta}(\widehat{g(D)(x)}).$ Therefore,
$$\lambda_{\theta}(\widehat{g(D)(x)})=\lambda_{\theta}(g\widehat{x}).$$ 
So, $\widehat{g(D)(x)}=g\cdot\widehat{x}.$
This concludes the proof.
\end{proof}

The following result shows the adjoint of the Fourier multiplier defined by \eqref{Fourier-multiplier} on the quantum Euclidean space.
\begin{lem}\label{Duality}Let $1< p,q< \infty$ and let $g(D):L^p(\mathbb{R}^d_\theta)\to L^q(\mathbb{R}^d_\theta)$ be the Fourier multiplier defined by \eqref{Fourier-multiplier} with the symbol $g.$ Then its adjoint $g(D)^*:L^{q'}(\mathbb{R}^d_\theta)\to L^{p'}(\mathbb{R}^d_\theta)$ equals to $\overline{g}(D),$ where $\overline{g}$ is the complex conjugate of the function $g,$ in the sense of Definition \ref{def-adj}.  
\end{lem}

\begin{proof} Let $x,y\in \mathcal{S}(\mathbb{R}^d_{\theta}).$  Then by definition of adjoint operator we can write
\begin{eqnarray*}
( x,g(D)^*(y))\overset{\text{(\ref{def-dul-operator})}}{=}\tau_{\theta}(x(g(D)^*(y))^*)=\tau_{\theta}(g(D)(x)y^*)\overset{\text{(\ref{def-dul-operator})}}{=}(g(D)(x), y).    
\end{eqnarray*}
By formulas  \eqref{def-dul-operator},   \eqref{Plancherel-relation}, and \eqref{F_trans_Fourier_trans}, we have
\begin{eqnarray*}
( x,g(D)^*(y))&\overset{\text{(\ref{def-dul-operator})}}{=}&\tau_{\theta}(x(g(D)^*(y))^*)=\tau_{\theta}(g(D)(x)y^*)\\
&\overset{\text{(\ref{Plancherel-relation}) }}{=}& \int_{\mathbb{R}^d} \widehat{g(D)(x)}(s)\overline{\widehat{y}(s)}ds\\
&\overset{\text{(\ref{F_trans_Fourier_trans}) }}{=} &\int_{\mathbb{R}^d} g(s)\widehat{x}(s)\overline{\widehat{y}(s)}ds\\
&=&\int_{\mathbb{R}^d} \widehat{x}(s)\overline{\overline{g(s)}\widehat{y}(s)}ds\\
&\overset{\text{(\ref{F_trans_Fourier_trans}) }}{=}&
\int_{\mathbb{R}^d} \widehat{x}(s)\overline{\widehat{\overline{g}(D)(y)}(s)}ds\\
&\overset{\text{(\ref{Plancherel-relation})}}{=}& \tau_{\theta}(x(\overline{g}(D)(y))^*)=( x,\overline{g}(D)(y)), \quad x,y\in\mathcal{S}(\mathbb{R}^d),
\end{eqnarray*}
where 
$ \overline{g}(D)(y):=\int_{\mathbb{R}^d} \overline{g(s)}\widehat{y}(s)U_{\theta}(s)ds.$

Since $\mathcal{S}(\mathbb{R}^d_{\theta})$ is dense in $L^p(\mathbb{R}^d_\theta)$ (see \cite[Remark 3.14, p. 515]{MSX}) for $1\leq p\leq \infty$, we have  $T^*_{g}=T_{\overline{g}}.$
\end{proof}
The next theorem is the main result of this section, which shows the $L^p\to L^q$ boundedness of the Fourier multiplier on the quantum Euclidean space. This is a quantum  analogue of \cite[Theorem 1.11]{Hor} on the quantum  Euclidean space.
\begin{thm}\label{H-multiplier theorem-th}(H\"ormander multiplier theorem). Let $1 < p \leq 2 \leq q < \infty.$ If $g:\mathbb{R}^d\to\mathbb{C}$ is a measurable function such that
$$\sup\limits_{t>0}t\left(\int\limits_{|g(\xi)|\geq t}d\xi\right)^{\frac{1}{p}-\frac{1}{q}}<\infty,$$
then the Fourier multiplier defined by \eqref{Fourier-multiplier} with the symbol $g$ is extended to a bounded map from $L^p(\mathbb{R}^d_\theta)$ to  $L^q(\mathbb{R}^d_\theta),$ and we have
\begin{eqnarray}\label{H-multiplier theorem}
\|g(D)\|_{L^p(\mathbb{R}^d_\theta) \rightarrow L^q(\mathbb{R}^d_\theta)}\lesssim \sup\limits_{t>0}t\left(\int\limits_{|g(\xi)|\geq t}d\xi\right)^{\frac{1}{p}-\frac{1}{q}}.
 \end{eqnarray} 
\end{thm}

\begin{proof} By duality it is sufficient to study two cases: $1<p\leq q'\leq 2$ and $1<q'\leq p \leq 2,$ where $1=\frac{1}{q}+\frac{1}{q'}.$ Also, since $\mathcal{S}(\mathbb{R}^d_{\theta})$ is dense in $L^p(\mathbb{R}^d_\theta)$ for any $1\leq p< \infty$ (see \cite{GJP}), it is sufficient to prove the result for any $x$ in $\mathcal{S}(\mathbb{R}^d_{\theta}).$ 

First, we consider the case $1<p\leq q'\leq 2,$ where $1=\frac{1}{q}+\frac{1}{q'}.$ By Lemma \ref{F_trans_Fourier_trans} we have
 \begin{equation}\label{neccesay-lemma}\widehat{g(D)(x)}=g\cdot\widehat{x}, \quad x\in \mathcal{S}(\mathbb{R}^d_{\theta}).
 \end{equation}
Then, it follows from Proposition \ref{R-H-Y_ineq} and \eqref{neccesay-lemma}  that  
\begin{eqnarray}\label{additive5.3}
\|g(D)(x)\|_{L^q(\mathbb{R}^d_\theta)} \overset{\eqref{R_H-Y_ineq}}{\leq} \|\widehat{g(D)(x)}\|_{L^{q'}(\mathbb{R}^d)} 
 \overset{ \eqref{neccesay-lemma}}{=} \|g\widehat{x}\|_{L^{q'}(\mathbb{R}^d)}.  
  \end{eqnarray}
  
Therefore, if we set $q':=r$ and $\frac{1}{s}:=\frac{1}{p}-\frac{1}{q}=\frac{1}{q'}-\frac{1}{p'},$ then for $h(\xi):=|g(\xi)|^{s}, \xi\in \mathbb{R}^{d},$ we are in a position to apply the Hausdorff-Young-Paley inequality in Theorem \ref{D_H-Y-P_ineq}. In other words, we obtain
 \begin{eqnarray}\label{additive5.4}
\left(\int_{\mathbb{R}^d} \Big(|\widehat{x}(\xi)|\cdot|g(\xi)|\Big)^{q'}d\xi\right)^\frac{1}{q'}\lesssim M^\frac{1}{s}_{|g|^{s}}\|x\|_{L^p(\mathbb{R}^d_\theta)}
 \end{eqnarray} 
for any $x\in L^p(\mathbb{R}^d_\theta).$ 
Let us study $M^\frac{1}{s}_{|g|^{s}}$ separately. Indeed, by definition 
$$M^\frac{1}{s}_{|g|^{s}}:=\left(\sup\limits_{t>0}t\int\limits_{|g(\xi)|^{s}\geq t}d \xi\right)^{\frac{1}{s}}=\left(\sup\limits_{t>0}t\int\limits_{  |g(\xi)|\geq t^{\frac{1}{s}}}d \xi\right)^{\frac{1}{s}}=\left(\sup\limits_{t>0}t^{s}\int\limits_{|g(\xi)|\geq t}d \xi\right)^{\frac{1}{s}}.$$ Since $\frac{1}{s}:=\frac{1}{p}-\frac{1}{q},$ it follows that
\begin{eqnarray}\label{M_g_estimate}
M^\frac{1}{s}_{|g|^{s}}&=&\left(\sup\limits_{t>0}t^{s}\int\limits_{|g(\xi)|\geq t}d \xi\right)^{\frac{1}{p}-\frac{1}{q}}\nonumber\\&=&\sup\limits_{t>0}t^{s\big(\frac{1}{p}-\frac{1}{q}\big)}\left(\int\limits_{|g(\xi)|\geq t}d \xi\right)^{\frac{1}{p}-\frac{1}{q}}\nonumber\\&=&\sup\limits_{t>0}t\left(\int\limits_{|g(\xi)|\geq t}d \xi\right)^{\frac{1}{p}-\frac{1}{q}}.
 \end{eqnarray}

Hence, combining (\ref{additive5.3}), (\ref{additive5.4}), and \eqref{M_g_estimate} we obtain
\begin{eqnarray*}
\|g(D)(x)\|_{L^q(\mathbb{R}^d_\theta)}  
 &\overset{\text{(\ref{additive5.3})}}{\lesssim}&\left(\int\limits_{\mathbb{R}^d}\Big(|\widehat{x}(\xi)|\cdot|g(\xi)|\Big)^{q'}d\xi \right)^\frac{1}{q'}  \\&\overset{ \text{(\ref{additive5.4})}} {\lesssim}& M^\frac{1}{s}_{|g|^{s}}\|x\|_{L^p\left(\mathbb{R}^d_\theta\right)} \\&\overset{ \text{\eqref{M_g_estimate}}}{=} & \sup\limits_{t>0}t\left(\int\limits_{|g(\xi)|\geq t}d \xi\right)^{\frac{1}{p}-\frac{1}{q}}\|x\|_{L^p(\mathbb{R}^d_\theta)},  
\end{eqnarray*}
for $1 < p \leq q'\leq 2.$ 

Next, we consider the case $q'\leq p\leq2$ so that $p'\leq (q')'=q,$ where $1=\frac{1}{q}+\frac{1}{q'}$ and $1=\frac{1}{p}+\frac{1}{p'}.$ Thus, the $L^p$-duality (see Lemma \ref{Duality}) yields that $g(D)^{*}=\overline{g}(D)$ and
\begin{eqnarray*}
 \|g(D)\|_{L^p(\mathbb{R}^d_\theta) \rightarrow L^q(\mathbb{R}^d_\theta)}= \|\overline{g}(D)\|_{L^{q'}(\mathbb{R}^d_\theta) \rightarrow L^{p'}(\mathbb{R}^d_\theta)}.  
\end{eqnarray*}
The symbol of the adjoint operator $g(D)^{*}$ equals to $\overline{g}$ and
obviously we have $|g(\xi)|=|\overline{g(\xi)}|$ for all $\xi\in\mathbb{R}^{d}.$ Set $\frac{1}{p}-\frac{1}{q}=\frac{1}{s}=\frac{1}{q'}-\frac{1}{p'}.$ Hence, by repeating the argument in the previous case we have
\begin{eqnarray*}
\|\overline{g}(D)(x)\|_{L^{p'}(\mathbb{R}^d_\theta)}  
 &\lesssim& \sup\limits_{t>0}t\left(\int\limits_{|\overline{g(\xi)}|\geq t}d \xi\right)^{\frac{1}{q'}-\frac{1}{p'}}\|x\|_{L^{q'}(\mathbb{R}^d_\theta)}\\
 &=& \sup\limits_{t>0}t\left(\int\limits_{|g(\xi)|\geq t}d \xi\right)^{\frac{1}{p}-\frac{1}{q}}\|x\|_{L^{q'}(\mathbb{R}^d_\theta)}.  
\end{eqnarray*}
In other words, we have 
$$\|\overline{g}(D)\|_{L^{q'}(\mathbb{R}^d_\theta) \rightarrow L^{p'}(\mathbb{R}^d_\theta)}\lesssim
\sup\limits_{t>0}t\left(\int\limits_{|g(\xi)|\geq t}d\xi\right)^{\frac{1}{p}-\frac{1}{q}}.$$ Combining both cases, we obtain
$$\|g(D)\|_{L^{p}(\mathbb{R}^d_\theta) \rightarrow L^{q}(\mathbb{R}^d_\theta)}\lesssim
\sup\limits_{t>0}t\left(\int\limits_{|g(\xi)|\geq t}d\xi\right)^{\frac{1}{p}-\frac{1}{q}}$$
for all $1 < p \leq 2 \leq q < \infty.$ This concludes the proof.
\end{proof}

As an application of Theorem \ref{H-multiplier theorem-th} we will show the $L^p -L^q$ estimate for the heat semi-group on quantum Euclidean spaces. 
Let us define the heat semi-groups
of operators on $L^2(\mathbb{R}_{\theta}^{d})$ by 
$t\mapsto e^{t\Delta_{\theta}},$ where $\Delta_{\theta}$ is the Laplacian defined by \eqref{laplacian}.
This operator can be defined as the Fourier multiplier. Observe that by definition (see \eqref{Fourier-multiplier}), it is an operator of the form $g(D)$ with the symbol $g(\xi)=e^{-t|\xi|^2}, \,\xi \in \mathbb{R}^d, \, t>0.$ Then a simple application of Theorem \ref{H-multiplier theorem-th} gives us the $L^p -L^q$ estimate of the heat semi-group.
\begin{cor}\label{heat-kernel}
\label{heat-semigroup} Let $1 < p \leq 2 \leq q < \infty.$ We have 
$$\|e^{t\Delta_{\theta}}\|_{L^{p}(\mathbb{R}^d_\theta) \rightarrow L^{q}(\mathbb{R}^d_\theta)}\leq c_{p,q} t^{-\frac{d}{2}\big(\frac{1}{p}-\frac{1}{q}\big)},\,\ t>0.
$$
\end{cor}
\begin{proof}By Theorem \ref{H-multiplier theorem-th} we have
 $$\|e^{t\Delta_{\theta}}\|_{L^{p}(\mathbb{R}^d_\theta) \rightarrow L^{q}(\mathbb{R}^d_\theta)}\leq c_{p,q}\cdot\sup\limits_{s>0}s\left(\int\limits_{  e^{-t|\xi|^2}\geq s}d\xi\right)^{\frac{1}{p}-\frac{1}{q}}.$$ 
 Then a simple calculation shows that the right hand side is equivalent to the function
 $t^{-\frac{d}{2}\big(\frac{1}{p}-\frac{1}{q}\big)}, \, t>0.$ This completes the proof.
\end{proof}

\section{Another approach for $L^p–L^q$ boundedness}\label{5s}
In this section, we obtain the $L^p -L^q$ boundedness of the Fourier multipliers on quantum Euclidean space with a different approach. In other words, we present another proof of H\"ormander multiplier Theorem \ref{H-multiplier theorem-th}. The idea comes from \cite{Z}, where the author obtained similar results in the context of so called locally compact quantum groups. 

\begin{lem}\label{L.5.1} (Hausdorff-Young inequality).   Let $1\leq{p}\leq2$  and $\frac{1}{p}+\frac{1}{p'}=1$. Then for any $f\in L^{p,p'}(\mathbb{R}^d)$ we have
\begin{eqnarray}\label{H-L_ineq}
\|\lambda_{\theta}(f)\|_{L^{p'}(\mathbb{R}_{\theta}^d )}  \lesssim   \|f\|_{L^{p,p'}(\mathbb{R}^d)}.
\end{eqnarray}
\end{lem}
\begin{proof}   We  consider the   operator  $\lambda_{\theta}(\cdot) $ defined by the formula \eqref{def-integration}.    Clearly, it is linear due to the linearity of the integral. Furthermore, by formula \eqref{def-integration}, we obtain  
\begin{eqnarray*} 
\|\lambda_{\theta}(\cdot)\|_{L^\infty(\mathbb{R}_{\theta}^d )}  &\overset{ \text{(\ref{def-integration})} }{\leq} &\int_{\mathbb{R}^d}|f(t)|\|U_{\theta}(t)\|_{L^\infty(\mathbb{R}_{\theta}^d )} dt 
=\int_{\mathbb{R}^d}|f(t)| dt= \|f\|_{L^1(\mathbb{R}^d)},\quad f\in L^{1}(\mathbb{R}^{d}). 
\end{eqnarray*}  
Moreover, by the Plancherel identity \eqref{Plancherel}, we have    
\begin{equation*} 
\|\lambda_{\theta}(f)\|_{L^{2}(\mathbb{R}^{d}_{\theta})}=\|f\|_{L^{2}(\mathbb{R}^{d})},\quad f\in L^{2}(\mathbb{R}^{d}).
\end{equation*}
Thus, $\lambda_{\theta}(\cdot) $ is bounded from  $L^{1,1}(\mathbb{R}^{d}_{\theta})$ to $L^{\infty,\infty}(\mathbb{R}^{d})$ and from $L^{2,2}(\mathbb{R}^d_\theta)$ to $L^{2,2}(\mathbb{R}^d).$ Hence, the assertion follows from the noncommutative Marcinkiewicz theorem (see \cite[Theorem 7.8.2, p. 434]{DPS}).  
\end{proof}

\begin{lem}\label{L.5.2} Let $1\leq p\leq2$ with $\frac{1}{p}+\frac{1}{p'}=1.$ Then for any $x\in L^{p}(\mathbb{R}_{\theta}^d)$ we have
 \begin{eqnarray}\label{R-H-Y_ineq}
\|\widehat{x}\|_{L^{p',p}(\mathbb{R}^d)} \leq \|x\|_{L^{p}(\mathbb{R}_{\theta}^d )}.
\end{eqnarray}
 \end{lem}
\begin{proof}The idea of the proof is similar to that of Theorem \ref{H-L_ineq}. But, the main question is to define the operator correctly. Let us define the map $T$ by
$$
T: x\mapsto \widehat{x}:=f,
$$
which is well defined on $L^{1}(\mathbb{R}^d)$ into $L^{\infty}(\mathbb{R}^d_\theta).$ Indeed, Applying the noncommutative H\"{o}lder inequality with respect to $p=1$ and $p'=\infty,$ we have
\begin{eqnarray*}
\|\widehat{x}\|_{L^{\infty}(\mathbb{R}^d)}=\ess\sup\limits_{t\in\mathbb{R}^d}|\tau_{\theta}(xU_{\theta}(t)^{*}|\leq \|x\|_{L^1(\mathbb{R}_{\theta}^d)}\|U_{\theta}(t)\|_{L^{\infty}(\mathbb{R}^d_\theta)}=\|x\|_{L^1(\mathbb{R}_\theta^d)}. 
\end{eqnarray*}
On the other hand, Plancherel’s (Parseval’s) identity \eqref{Plancherel} exhibits an isometry $T$ from $L^2(\mathbb{R}^d_\theta)$ onto $L^2(\mathbb{R}^d)$ by $T(x)=\widehat{x}$ (see, Remark \ref{Remark}). Interpolating as in Lemma  \ref{L.5.1}, we obtain that 
$T$ is extended to a linear bounded map from $L^p(\mathbb{R}^d_\theta)$ to $L^{p',p}(\mathbb{R}^d).$ 
In other words, if $x\in L^{p}(\mathbb{R}_{\theta}^d),$ then $\widehat{x}=f$  belongs to $L^{p',p}(\mathbb{R}^d),$ and we have
$$
\|\widehat{x}\|_{L^{p',p}(\mathbb{R}^d)}\leq \|x\|_{L^{p}(\mathbb{R}_{\theta}^d)},
$$
thereby completing the proof.
\end{proof}

\begin{thm} Let  $1<p\leq2\leq q<\infty$ and $g\in L^{r,\infty}(\mathbb{R}^d)$ with $\frac{1}{r}= \frac{1}{p} - \frac{1}{q}.$ Then the Fourier multiplier defined by \eqref{Fourier-multiplier} is extended to a bounded map from $L^p(\mathbb{R}^d_\theta)$ to  $L^q(\mathbb{R}^d_\theta),$ and we have
$$
\|g(D)\|_{L^p(\mathbb{R}^d_\theta) \rightarrow L^q(\mathbb{R}^d_\theta)}\lesssim \|g\|_{L^{r,\infty}(\mathbb{R}^d)}.$$ 
\end{thm}
\begin{proof} Since $\frac{1}{r}:=\frac{1}{p} - \frac{1}{q},$ we have that $\frac{1}{q'}=\frac{1}{r} +\frac{1}{p'}.$  By definition of the Fourier multiplier \eqref{Fourier-multiplier}
$$g(D)(x)=\lambda_{\theta}(g\cdot \widehat{x}).$$
Then for any $g\in L^{r,\infty}(\mathbb{R}^d_\theta)$ and
$x\in L^p(\mathbb{R}^d_\theta),$ by Lemmas \ref{L.5.1} and \ref{L.5.2}, and Proposition \ref{G-Prop}, we have
$$\|g(D)(x)\|_{L^q(\mathbb{R}^d_\theta)}= \|\lambda_{\theta}(g\cdot \widehat{x})\|_{L^q(\mathbb{R}^d_\theta)}\overset{\eqref{H-L_ineq}}{\lesssim}\|g\cdot \widehat{x}\|_{L^{q',q}(\mathbb{R}^d)}\overset{\eqref{Lorentz-embedding-2}}{\lesssim}\|g\|_{L^{r,\infty}(\mathbb{R}^d)}\|\widehat{x}\|_{L^{p',q}(\mathbb{R}^d)}$$
$$\overset{\eqref{Lorentz-embedding-1}}{\lesssim}\|g\|_{L^{r,\infty}(\mathbb{R}^d)}\|\widehat{x}\|_{L^{p',p}(\mathbb{R}^d)}\overset{\eqref{R-H-Y_ineq}}{\lesssim}\|g\|_{L^{r,\infty}(\mathbb{R}^d)}\|x\|_{L^{p}(\mathbb{R}_{\theta}^d)},$$
thereby completing the proof.
\end{proof}

\section{Sobolev and logarithmic Sobolev inequalities on quantum Euclidean spaces}\label{sc8}
In this section, we obtain Sobolev and logarithmic Sobolev type inequalities on quantum Euclidean spaces. 

As another consequence of Theorem \ref{H-multiplier theorem-th}, we obtain the following Sobolev type embedding.  
\begin{thm}\label{sobolev-embed}
\label{heat-semigroup} Let $1 < p \leq 2 \leq q < \infty.$ Then for any $x\in L_{s}^{p}(\mathbb{R}^d_\theta)$ the inequality  
$$
\|x\|_{L^{q}(\mathbb{R}^d_\theta)}\leq c_{p,q,s} \|x\|_{L_{s}^{p}(\mathbb{R}^d_\theta)},
$$
where $c_{p,q,s}>0$ is a constant independent of $x,$ holds true provided that
$s\geq d\left(\frac{1}{p}-\frac{1}{q}\right),$ $p\neq{q}.$ 
\end{thm}
\begin{proof} Let us consider $g(\xi)=(1+|\xi|^2)^{- s/2}, \quad \xi\in\mathbb{R}^d$ and $s>0.$ Take $g(D)=(1-\Delta_{\theta})^{- s/2}.$ In order to use Theorem \ref{H-multiplier theorem-th}, we need first to calculate the right hand side of \eqref{H-multiplier theorem}. For this, we have
\begin{eqnarray*}
\sup\limits_{t>0}t\left(\int\limits_{|g(\xi)|\geq t}d\xi\right)^{\frac{1}{p}-\frac{1}{q}}&=& \sup\limits_{ t>0}t\left(\nu\{\xi\in\mathbb{R}^d: \frac{1}{(1+|\xi|^2)^{-s/2}}\geq t\}\right)^{\frac{1}{p}-\frac{1}{q}},
\end{eqnarray*}
where $\nu$ is the Lebesgue measure on $\mathbb{R}^d.$ 
Since 
\begin{eqnarray*}
\nu\{\xi\in\mathbb{R}^d: \frac{1}{(1+|\xi|^2)^{-s/2}}\geq t\}
= \begin{cases} 0, \,\ \text{if} \,\ 1<t<\infty, \\
\nu\{\xi\in\mathbb{R}^d: |\xi|\leq \sqrt{t^{-2/s}-1}\},\,\ \text{if} \,\ 0<t\leq 1,
    \end{cases}
\end{eqnarray*}
it follows that
\begin{eqnarray*}
\sup\limits_{t>0}t\left(\int\limits_{|g(\xi)|\geq t}d\xi\right)^{\frac{1}{p}-\frac{1}{q}}&=& \sup\limits_{0<t\leq 1}t\left(\nu\{\xi\in\mathbb{R}^d: \frac{1}{(1+|\xi|^2)^{-s/2}}\geq t\}\right)^{\frac{1}{p}-\frac{1}{q}}\\
&=& \sup\limits_{0<t\leq 1}t\left(\nu\{\xi\in\mathbb{R}^d: |\xi|\leq \sqrt{t^{-2/s}-1}\}\right)^{\frac{1}{p}-\frac{1}{q}}\\
&\leq&\sup\limits_{0<t\leq 1}t\left(\nu\{\xi\in\mathbb{R}^d: |\xi|\leq t^{-1/s}\}\right)^{\frac{1}{p}-\frac{1}{q}}\\
&\lesssim&\sup\limits_{0<t\leq 1}t^{1-\frac{d}{s}(\frac{1}{p}-\frac{1}{q})}<\infty,
\end{eqnarray*}
whenever $s\geq d\left(\frac{1}{p}-\frac{1}{q}\right), p\neq q.$ 
Thus, the function $g$ satisfies the assumptions of Theorem \ref{H-multiplier theorem-th}, and consequently, we have \begin{eqnarray*}
\|x\|_{L^{q}(\mathbb{R}^d_\theta)}&=& \|(1-\Delta_{\theta})^{-s/2}(1-\Delta_{\theta})^{s/2}x\|_{L^{q}(\mathbb{R}^d_\theta)}\\
&\overset{\eqref{H-multiplier theorem}}{\leq}& c_{p,q,s} \sup\limits_{t>0}t\left(\int\limits_{|g(\xi)|\geq t}d\xi\right)^{\frac{1}{p}-\frac{1}{q}}\|(1-\Delta_{\theta})^{s/2}x\|_{L^{p}(\mathbb{R}^d_{\theta})}\\
&\leq &c_{p,q,s}\|x\|_{L_{s}^{p}(\mathbb{R}^d_\theta)}
\end{eqnarray*}
for $s\geq d\left(\frac{1}{p}-\frac{1}{q}\right)$ and $x\in L_{s}^{p}(\mathbb{R}^d_\theta),$
thereby completing the proof.
\end{proof} 
We also obtain the following simple consequence of the noncommutative H\"{o}lder inequality.
\begin{lem}\label{inter-lem}Let $1\leq p \leq r \leq  q\leq \infty$   be such that $\frac{1}{r}=\frac{\eta}{p}+\frac{1-\eta}{q}$ for some $\eta\in(0,1).$ Then for any $x\in L^q(\mathbb{R}^d_\theta)\cap L^p(\mathbb{R}^d_\theta),$ we have
\begin{equation}\label{inter-lem1}
\|x\|_{L^r(\mathbb{R}^d_\theta)}\leq \|x\|^{\eta}_{L^p(\mathbb{R}^d_\theta)}\|x\|^{1-\eta}_{L^q(\mathbb{R}^d_\theta)}.
 \end{equation}
\end{lem}
\begin{proof} From the noncommutative H\"{o}lder inequality \cite[Theorem 5.2.9]{Xu2008} with 
$\frac{1}{p/r\eta}+\frac{1}{q/(1-\eta)r}=1,$   we get
\begin{eqnarray*}
\|x\|^r_{L^r(\mathbb{R}^d_\theta)}=\tau_{\theta}(|x|^{r\eta}|x|^{(1-\eta)r}) 
\leq \|x\|^{r\eta}_{L^p(\mathbb{R}^d_\theta)}\|x\|^{(1-\eta)r}_{L^q(\mathbb{R}^d_\theta)}.  
\end{eqnarray*} 
This shows \eqref{inter-lem1}, thereby completing the proof.
\end{proof}

\begin{lem}(Logarithmic H\"{o}lder inequality)\label{log-holder}
 Let $1\leq p< q< \infty.$ Then for any $0\neq x\in L^p(\mathbb{R}^d_\theta)\cap L^q(\mathbb{R}^d_\theta)$ we have    \begin{equation}\label{log-estimate}
\tau_{\theta}\left(\frac{|x|^p}{\|x\|^p_{L^p(\mathbb{R}^d_\theta)}}\log \Big(\frac{|x|^p}{\|x\|^p_{L^p(\mathbb{R}^d_\theta)}}\Big)\right)\leq \frac{q}{q-p}\log \Big(\frac{\|x\|^p_{L^q(\mathbb{R}^d_\theta)}}{\|x\|^p_{L^p(\mathbb{R}^d_\theta)}}\Big).
 \end{equation}
\end{lem}
\begin{proof} The proof is a modification of the proof in \cite{CAM2023}.  Let us consider the 
function 
$$f(\frac{1}{r})=\log (\|x\|_{L^r(\mathbb{R}^d_\theta)}), \quad r>0.$$
It follows from Lemma \ref{inter-lem} with  $r>1$ and $\eta\in (0,1)$ such that $\frac{1}{r}=\frac{\eta}{p}+\frac{1-\eta}{q}$ that 
\begin{eqnarray*}f(\frac{1}{r})&=&\log (\|x\|_{L^r(\mathbb{R}^d_\theta)}) \overset{\text{(\ref{inter-lem})}}{\leq}  \log (\|x\|^{\eta}_{L^p(\mathbb{R}^d_\theta)}\|x\|^{1-\eta}_{L^q(\mathbb{R}^d_\theta)})\\
&=&\log (\|x\|^{\eta}_{L^p(\mathbb{R}^d_\theta)})+\log (\|x\|^{1-\eta}_{L^q(\mathbb{R}^d_\theta)})\\
&=&\eta f(\frac{1}{p})+(1-\eta) f(\frac{1}{q}), 
\end{eqnarray*}
which implies that the function $f$ is convex. Since
$$f(s)=s\log (\tau_{\theta}(|x|^{1/s})),$$ taking derivative with respect to $s$ (see \cite[Lemma 4.2]{Xiong}), we obtain
\begin{eqnarray*}f'(s)&=&\log (\tau_{\theta}(|x|^{1/s}))+s\Big(\log (\tau_{\theta}(|x|^{1/s}))\Big)'_{s}\\
&=&\log (\tau_{\theta}(|x|^{1/s}))+s\frac{\Big(\tau_{\theta}(|x|^{1/s})\Big)'_{s}}{\tau_{\theta}(|x|^{1/s})}\\
&=&\log (\tau_{\theta}(|x|^{1/s}))-\frac{1}{s}\frac{\tau_{\theta}(|x|^{1/s}\log(|x|))}{\tau_{\theta}(|x|^{1/s})}. 
\end{eqnarray*}
On the other hand, since $f$ is a differentiable convex function of $s$, it follows that
$$f'(s)\geq \frac{f(s_0)-f(s)}{s_0-s}, \quad s_0>s>0.$$
Taking $s=\frac{1}{p}$ and $s_0=\frac{1}{q},$ we have
\begin{equation}\label{almost-log-holder}
    p\frac{\tau_{\theta}(|x|^{p}\log(|x|))}{\tau_{\theta}(|x|^{p})}-\log (\tau_{\theta}(|x|^{p}))
\leq\frac{pq}{q-p}\log \Big(\frac{\|x\|_{L^q(\mathbb{R}^d_\theta)}}{\|x\|_{L^p(\mathbb{R}^d_\theta)}}\Big).
\end{equation}
But, by linearity of $\tau_{\theta}$ the left hand side of the previous inequality is represented as follows
\begin{eqnarray*}&p&\frac{\tau_{\theta}(|x|^{p}\log(|x|))}{\tau_{\theta}(|x|^{p})}-\log (\tau_{\theta}(|x|^{p}))=p\frac{\tau_{\theta}(|x|^{p}\log(|x|))}{\tau_{\theta}(|x|^{p})}-\frac{\tau_{\theta}(|x|^{p})\cdot\log (\tau_{\theta}(|x|^{p}))}{\tau_{\theta}(|x|^{p})}\\
&=&\frac{\tau_{\theta}(|x|^{p}\log(|x|^p))}{\tau_{\theta}(|x|^{p})}-\frac{\tau_{\theta}\Big(|x|^{p}\log (\|x\|^{p}_{L^p(\mathbb{R}^d_\theta)})\Big)}{\tau_{\theta}(|x|^{p})}\\
&=&\frac{\tau_{\theta}(|x|^{p}\log(|x|^p))-\tau_{\theta}\Big(|x|^{p}\log (\|x\|^{p}_{L^p(\mathbb{R}^d_\theta)})\Big)}{\tau_{\theta}(|x|^{p})}\\
&=&\frac{\tau_{\theta}\Big(|x|^{p}\Big(\log(|x|^p)-\log (\|x\|^{p}_{L^p(\mathbb{R}^d_\theta)})\Big)\Big)}{\tau_{\theta}(|x|^{p})}=\tau_{\theta}\left(\frac{|x|^p}{\|x\|^p_{L^p(\mathbb{R}^d_\theta)}}\log \Big(\frac{|x|^p}{\|x\|^p_{L^p(\mathbb{R}^d_\theta)}}\Big)\right).
\end{eqnarray*}
Hence, the assertion follows from \eqref{almost-log-holder}.
\end{proof}
As an application of Theorem \ref{sobolev-embed} for the case $1 < p \leq 2 \leq p^* < \infty$ we obtain the following logarithmic type Sobolev inequality. 
\begin{thm}(Logarithmic Sobolev inequality)\label{sobolev-ineq} Let  $1<p< 2$ and   $d(\frac{2-p}{2p})\leq s<\frac{d}{p}.$
Then for any $x\in L_{s}^p(\mathbb{R}^d_\theta)$ we have the fractional
logarithmic Sobolev  inequality
\begin{equation}\label{sobolev-log-estimate}
\tau_{\theta}\left(\frac{|x|^p}{\|x\|^p_{L^p(\mathbb{R}^d_\theta)}}\log \Big(\frac{|x|^p}{\|x\|^p_{L^p(\mathbb{R}^d_\theta)}}\Big)\right)\leq \frac{d}{sp}\log \Big(C_{p,d}\frac{\|x\|^p_{L_s^p(\mathbb{R}^d_\theta)}}{\|x\|^p_{L^p(\mathbb{R}^d_\theta)}}\Big),  \,\ x\in L_{s}^p(\mathbb{R}^d_\theta),
 \end{equation}
where $C_{p,d}>0$ is a constant intependent of $x.$
\end{thm}
\begin{proof}
By Theorem \ref{sobolev-embed} we have $L_s^p(\mathbb{R}^d_\theta)\subset L^{p^*}(\mathbb{R}^d_\theta)$ with $1<p\leq 2\leq p^*<\infty$ and $s\geq \frac{d}{p}-\frac{d}{p^*}$ such that $p\neq p^*.$ This means that there exists a constant $C_{p,p^*,d}>0$ such that
$$
\|x\|_{L^{p^*}(\mathbb{R}^d_\theta)}\leq C_{p,p^*,d}\|x\|_{L_s^p(\mathbb{R}^d_\theta)}, \,\ x\in L_s^p(\mathbb{R}^d_\theta). 
$$
Since $p^*=\frac{dp}{d-sp}\geq2$ by the assumption, taking $q:=p^*$ in Lemma \ref{log-holder} and since $\frac{p^*}{p^* -q}=\frac{d}{sp},$ the assertion follows from \eqref{log-estimate}.
\end{proof}

\subsection{Nash-type inequality on the quantum Euclidean space}
In this subsection, we prove the Nash-type inequality on the quantum Euclidean space. In the classical case, this inequality is applied as a main tool in computing the decay rate for the heat equation for the sub-Laplacian.
\begin{thm}\label{Nash-inequality}(Nash-type inequality). Let $x\in W^{1,2}(\mathbb{R}^d_\theta)$ and $d>2$. Then  we have 
 \begin{eqnarray}\label{Nash's-inequality}
  \|x\|^{1+\frac{2}{d}}_{L^{2}(\mathbb{R}^d_\theta)} \leq C_{1,d}\| x\|_{W^{1,2}(\mathbb{R}^d_\theta)}\|x\|^{\frac{2}{d}}_{L^{1}(\mathbb{R}^d_\theta)},  
 \end{eqnarray}   
  where $C_{1,d}>0$ is a constant independent of $x.$
\end{thm}

\begin{proof} 
By the Sobolev inequality in Theorem \ref{sobolev-embed}, we get
\begin{eqnarray}\label{s-inequality}
  \|x\|_{L^{2^*}(\mathbb{R}^d_\theta)} \leq C_d \|x\|_{W^{1,2}(\mathbb{R}^d_\theta)},\quad 2^*=\frac{2d}{d-2},
 \end{eqnarray} 

 where $C_d>0$ is a constant independent of $x$.  

By Lemma \ref{inter-lem} with 
 $r=2$, $q=2^* = \frac{2d}{d-2},$  and $p=1,$ we obtain
\begin{eqnarray}\label{interpolation-inequality}
  \|x\|_{L^{2}(\mathbb{R}^d_\theta)} \leq \|x\|^{\eta}_{L^{1}(\mathbb{R}^d_\theta)}\|x\|^{1-\eta}_{L^{2^*}(\mathbb{R}^d_\theta)},
 \end{eqnarray} 
where
\begin{eqnarray}\label{parametr}
\eta:=\frac{\frac{1}{r}-\frac{1}{q}}{\frac{1}{p}-\frac{1}{q}}=\frac{\frac{1}{2}-\frac{d-2}{2d}}{\frac{d+2}{2d}}=\frac{2}{d+2}.
\end{eqnarray} 
Hence, combining    \eqref{s-inequality},  \eqref{interpolation-inequality}, and \eqref{parametr}, we get
\begin{eqnarray*}
\|x\|_{L^{2}(\mathbb{R}^d_\theta)} &\overset{ \text{\eqref{s-inequality} \eqref{interpolation-inequality}} }{\leq}  ( C_d \| x\|_{W^{1,2}(\mathbb{R}^d_\theta)})^{1-\eta}  \|x\|^{\eta}_{L^{1}(\mathbb{R}^d_\theta)} \overset{ \text{(\ref{parametr})} }{=} C^{\frac{d}{d+2}}_d \| x\|^{\frac{d}{d+2}} _{W^{1,2}(\mathbb{R}^d_\theta)} \|x\|^{\frac{2}{d+2}}_{L^{1}(\mathbb{R}^d_\theta)}. 
\end{eqnarray*}
 Rising both sides of the last inequality to power $\frac{d+2}{d}$, one gets the inequality \eqref{Nash's-inequality}.
\end{proof}

\begin{rem} Note that it is well-known in classical Euclidean spaces that the logarithmic Sobolev inequality is equivalent to the Nash inequality when $p=2$. First, we conjectured in the previous version of this paper that the equivalence remains true also in noncommutative setting. However, this was solved recently in an another paper by the authors (see \cite[Theorem 4.2]{RShT-JFA}. For more details about these inequalities and their applications in PDEs for Lie groups we refer the reader to \cite{CAM2023, GKR, Mer}.
\end{rem}

{\bf Conflict of interest.}
We can conceive of no conflict of
interest in the publication of this paper. The work has not
been published previously and it has not been submitted for publication elsewhere.

{\bf Data availability}
No new data were created or analysed during this study. Data sharing is not applicable to this article.

\section{Acknowledgements}

The work was partially supported by the grant No. AP23483532 of the Science Committee of the Ministry of Science and Higher Education of the Republic of Kazakhstan.
Authors would like to thank to  Dr. Edward McDonald for helpful discussions on non-commutative Euclidean spaces and showing us the proof of Lemma \ref{F_trans_Fourier_trans}. 
The authors were partially supported by Odysseus and Methusalem grants (01M01021 (BOF Methusalem) and 3G0H9418 (FWO Odysseus)) from Ghent Analysis and PDE center at Ghent University. The first author was also supported by the EPSRC grants EP/R003025/2 and EP/V005529/1.

Authors thank the anonymous referee for reading the paper and providing thoughtful comments, which improved the exposition of the paper.

\begin{center}

\end{center}

\end{document}